\documentclass[10pt,oneside,english]{amsart}
\usepackage{lmodern}
\usepackage[T1]{fontenc}
\usepackage[latin9]{inputenc}
\usepackage{geometry}
\usepackage{amsbsy}
\usepackage{amsthm}
\usepackage{amssymb}
\usepackage{psfrag}
\usepackage{amstext}
\usepackage{color}
\usepackage{verbatim}

\makeatletter
\numberwithin{equation}{section}
\numberwithin{figure}{section}
\theoremstyle{plain}
\newtheorem{thm}{\protect\theoremname}
  \theoremstyle{plain}
  \newtheorem{lem}[thm]{\protect\lemmaname}
\theoremstyle{plain}
\newtheorem{cor}[thm]{\protect\corollaryname}
\theoremstyle{plain}
\newtheorem{prop}[thm]{\protect\propositionname}
\theoremstyle{plain}
\newtheorem{remark}[thm]{\protect\remarkname}
\makeatother

\usepackage{babel}
\providecommand{\lemmaname}{Lemma}
\providecommand{\theoremname}{Theorem}
\providecommand{\corollaryname}{Corollary}
\providecommand{\propositionname}{Proposition}

\providecommand{\remarkname}{Remark}
\DeclareSymbolFont{matha}{OML}{txmi}{m}{it}

\makeatletter
\newsavebox\myboxA
\newsavebox\myboxB
\newlength\mylenA
\newcommand*\mybar[2][0.75]{%
    \sbox{\myboxA}{$\m@th#2$}%
    \setbox\myboxB\null
    \ht\myboxB=\ht\myboxA%
    \dp\myboxB=\dp\myboxA%
    \wd\myboxB=#1\wd\myboxA
    \sbox\myboxB{$\m@th\overline{\copy\myboxB}$}
    \setlength\mylenA{\the\wd\myboxA}
    \addtolength\mylenA{-\the\wd\myboxB}%
    \ifdim\wd\myboxB<\wd\myboxA%
       \rlap{\hskip 0.5\mylenA\usebox\myboxB}{\usebox\myboxA}%
    \else
        \hskip -0.5\mylenA\rlap{\usebox\myboxA}{\hskip 0.5\mylenA\usebox\myboxB}%
    \fi}
\makeatother

\newcommand\restr[2]{{
  \left.\kern-\nulldelimiterspace 
  #1 
  \vphantom{\big|} 
  \right|_{#2} 
  }}

\newcommand{\oPsi}{\mybar{\Psi}}
\newcommand{\nN}{\mathcal{N}_{q,1}}
\newcommand{\oN}{\mathcal{N}_{0,1}}
\newcommand{\bN}{\mybar{\mathcal{N}}}

\newcommand{\nTAP}{\operatorname{TAP}}

\newcommand{\e}{\mathbb{E}}
\newcommand{\p}{\mathbb{P}}
\newcommand{\Reals}{\mathbb{R}}

\global\long\def\bs{\boldsymbol{\sigma}}
\global\long\def\tbs{\tilde{\boldsymbol{\sigma}}}

\global\long\def\e{\mathbb{E}}
\global\long\def\p{\mathbb{P}}

\global\long\def\supp{{\operatorname{supp}}}

\global\long\def\PP{\mathcal{P}}

\newcommand{\eps}{{\varepsilon}}

\newcommand{\TAP}{\operatorname{TAP}}
\newcommand{\EA}{\operatorname{EA}}

\newcommand{\qea}{q_{\mbox{\rm \tiny EA}}}

\begin{document}

	\title[TAP free energy II]{The generalized TAP free energy II}
		\author{Wei-Kuo Chen}\thanks{School of Mathematics, University of Minnesota. Email: wkchen@umn.edu. Partially supported by NSF grant DMS-17-52184.}
			\author{Dmitry Panchenko}\thanks{Department of Mathematics. University of Toronto. Email: panchenk@math.toronto.edu. Partially supported by NSERC}
			\author{Eliran Subag}\thanks{Courant Institute. Email: esubag@cims.nyu.edu. Supported by the Simons Foundation.}
		
		\begin{abstract}	
			In a recent paper \cite{GenTAP}, we developed the generalized TAP approach  for mixed $p$-spin models with Ising spins at positive temperature. Here we extend these results in two directions. We find a simplified representation for the energy of the generalized TAP states in terms of the Parisi measure of the model and, in particular, show that the energy of all states at a given distance from the origin is the same. Furthermore, we prove the analogues of the positive temperature results at zero temperature, which concern the ground-state energy and the organization of ground-state configurations in space. 	
		\end{abstract}
		
	\maketitle			

\section{Introduction and main results} 

The TAP approach, named after Thouless, Anderson and Palmer, was originally introduced in \cite{TAP}, where their famous equations for the magnetization and representation for the free energy of the SK model were derived. 
In a recent paper \cite{GenTAP}, adopting ideas from \cite{SubagFEL}, we defined the generalized TAP free energy using a geometric approach for mixed $p$-spin models with Ising spins, at any positive temperature. Our first goal here will be to compute the energy of all generalized TAP states in terms of their distance to the origin. The main focus, however, will be on the zero temperature analogue of the analysis in \cite{GenTAP}. Of course, as the temperature tends to zero the Gibbs measure concentrates on near maximal energies,  hence this analysis deals with the ground state energy and configurations. In particular, the corresponding TAP representation at zero temperature expresses the ground state energy, and the location and structure of TAP states contain information about the organization of ground state configurations in space.

The first rigorous mathematical results concerning the TAP approach were derived by Talagrand \cite{Talbook1} who established the TAP equations for the SK model at high temperature; see also the works of Chatterjee \cite{chatt10} and Bolthausen \cite{EB,EB2}. Much more recently, an analogue of the TAP equations within pure states was proved for generic mixed $p$-spin models at low temperature by Auffinger and Jagannath \cite{AJ16}. Moreover, in \cite{CPTAP17}, the TAP representation for the free energy was  proved for general mixed models by the first two authors. In the setting of the spherical models, the representation for the free energy was proved for the $2$-spin model by Belius and Kistler \cite{BelKist}, and at very low temperature, for the $p$-spin model with $p\geq 3$ by the third author \cite{Subag17} and for mixed models close to pure by Ben Arous, Zeitouni and the third author \cite{BSZ}.

In all of those works, the analysis was done at the level of pure states. As the temperature tends to zero, they degenerate to a single point and the TAP correction converges to zero, leaving only the energy term in the representation for the free energy. As a result, the TAP approach at the level of pure states trivializes at zero temperature. 
In \cite{GenTAP,SubagFEL} the generalized TAP free energy was defined based on geometric principles, inspired by structural properties of the Gibbs measure,
consequent to the famous ultrametricity property \cite{M1,M2,MPV} proved by the second author in \cite{ultrametricity} (see also \cite{SKmodel}). In contrast to the above, in addition to the pure states, this approach also treats ancestral states and generalizes to zero temperature in a  natural way, as we shall see below.

\subsection{Previous results at positive temperature}
Let us introduce the model and recall the results from our previous paper \cite{GenTAP}. Since these results will be used to pass to the zero temperature limit, here we will also introduce an inverse temperature parameter $\beta>0$. The pure $p$-spin Hamiltonian indexed by $\bs\in \Sigma_N:=\{-1,1\}^N$ is defined by
		\begin{align}
		\label{hamp}
		H_{N,p}(\bs)=\frac{1}{N^{(p-1)/2}}\sum_{i_1,\ldots,i_p=1}^{N}g_{i_1,\ldots,i_p}\sigma_{i_1}\cdots\sigma_{i_p},
		\end{align}
where $g_{i_1,\ldots,i_p}$ are i.i.d. standard Gaussian random variables. Given a sequence $(\beta_p)_{p\geq 1}$ that decreases fast enough, for example, $\sum_{p\geq 1}2^p\beta_p^2<\infty,$
		\begin{align}
		\label{hamx}
		H_N(\bs)=\sum_{p\geq 1}\beta_p H_{N,p}(\bs)
		\end{align}
is called a mixed $p$-spin Hamiltonian. Here the processes $H_{N,p}$ are independent of each other for $p\geq 1.$ The covariance of the Gaussian process $H_N(\bs)$ equals
		\begin{align}
		\e H_N(\bs^1)H_N(\bs^2)=N\xi\bigl(R(\bs^1,\bs^2)\bigr),
		\end{align} 
		where $R(\bs^1,\bs^2)=\frac{1}{N}\sum_{i=1}^N\sigma_i^1\sigma_i^2$ is called the overlap of $\bs^1$ and $\bs^2$, and where
		\begin{align}
		\xi(s)=\sum_{p\geq 1}\beta_p^2s^p.
		\end{align} 
Let us recall the Parisi formula \cite{Parisi79, Parisi80} for the free energy
	\begin{align}
		F_N(\beta)&=\frac{1}{\beta N}\log \sum_{\bs\in \Sigma_N}e^{\beta H_N(\bs)}.
	\end{align}
	If $\mathcal{M}_{0,1}$ is the space of probability measures on $[0,1]$, for $\zeta\in \mathcal{M}_{0,1},$ let $\Phi_{\zeta}^\beta(t,x)$ be the solution on $[0,1]\times \Reals$ of the Parisi PDE
	\begin{equation}
	\partial_t \Phi_{\zeta}^\beta = -\frac{\beta^2\xi''(t)}{2}\Bigl(
	\partial_{xx} \Phi_{\zeta}^\beta + \zeta(t)\bigl(\partial_x \Phi_{\zeta}^\beta\bigr)^2
	\Bigr)
	\label{ParisiPDEOrig}
	\end{equation}
with the boundary condition $\Phi_{\zeta}(1,x)=\log 2\cosh x,$ where $\zeta(t):=\zeta([0,t]).$ Let
	\begin{equation}
	\PP_\beta(\zeta):= \Phi_{\zeta}^\beta(0,0)-\frac{\beta^2}{2}\int_0^{1}\!s\xi''(s)\zeta(s)\,ds.
	\end{equation}
	Then, the limit of the free energy is given by the Parisi formula \cite{Parisi79, Parisi80},
	\begin{equation}
	\lim_{N\to\infty}\e F_N(\beta) = 
	\frac{1}{\beta}\inf_{\zeta\in \mathcal{M}_{0,1}}\PP_\beta(\zeta)
	=\frac{1}{\beta}\PP_\beta(\zeta_{\beta}^*),
	\label{eqParisiOrig}
	\end{equation}
	which was first proved by Talagrand in \cite{Tal03} (building on a breakthrough by Guerra \cite{Guerra}), and later generalized to models with odd spin interactions in \cite{Pan00}.
	The minimizer $\zeta_{\beta}^*$ is unique \cite{AC15} (see also \cite{JT16}) and is called the {Parisi measure}.

Next, we recall the generalized TAP free energy at inverse-temperature $\beta>0$. For $m\in [-1,1]^N$ and $\eps>0,$ let us consider a {narrow band} of configurations $\bs\in\Sigma_N$ close to the hyperplane perpendicular to $m,$ 
\begin{equation}
B(m,\eps)=\Big\{\bs\in\Sigma_N:\,|R(\bs,m)-R(m,m)|=\frac{1}{N}|m\cdot (\bs-m)|<\eps \Big\}.
\end{equation}
Given $\delta>0$ and $n\geq 1,$ let us consider a set consisting of $n$ configurations in this narrow band $\bs^{1},\ldots,\bs^{n}\in B(m,\eps)$ such that all $\tbs^i = \bs^i-m$ are almost orthogonal to each other,
\begin{equation}
\label{eq:Bn}
B_n(m,\eps,\delta)
= \Big\{ (\bs^{1},\ldots,\bs^{n})\in B(m,\eps)^n:\,\forall i\neq j,\,\,\big|R({\bs}^{i},{\bs}^{j})-R(m,m)\big|<\delta\Big\}.
\end{equation}
For real numbers  $\eps,\delta>0$ and  an integer number $n\geq1$, let
\begin{equation}
\label{eq:TAPnbeta}
\nTAP^\beta_{N,n}(m,\eps,\delta) :=\frac{1}{nN}\log\sum_{B_{n}(m,\eps,\delta)}e^{\beta\sum_{i=1}^{n}\big[H_{N}(\bs^{i})-H_N(m)\big]}.
\end{equation} 
The motivation for this functional was given in \cite{GenTAP}, so we will not repeat it here.

We will denote the concave conjugate of the Parisi functional $\Phi_\zeta^\beta(q,x)$ defined in (\ref{ParisiPDEOrig}) by
\begin{equation}
\Lambda_\zeta^\beta(q,a):=\inf_{x\in\Reals}\Bigl(\Phi_\zeta^\beta(q,x)-ax\Bigr), \,\, a\in [-1,1].
\label{eqCCPar}
\end{equation}
For $a\in (-1,1)$, the minimizer on the right-hand side exists and is denoted by $\oPsi_\beta(q,a,\zeta).$ Let $M_*$ denote the space of probability measures
\begin{equation}
M_{*} = \Pr([-1,1]).
\end{equation}
For $\mu\in M_{*}$ such that $\int\! a^2\, d\mu(a)=q\in [0,1]$, we define
\begin{equation}
\nTAP^\beta(\mu,\zeta):=
\int\! \Lambda_{\zeta}^\beta(q,a)\,d\mu(a)-\frac{\beta^2}{2}\int_q^{1}\!s\xi''(s)\zeta(s)\,ds.
\label{eqTAPfirstZeta}
\end{equation}
Notice that this functional depends only on the values of $\zeta(s)$ on the interval $[q,1]$, so we can view it as a functional on the space $\mathcal{M}_{q,1}$ of all cumulative distribution functions on $[q,1].$ Finally, define the TAP functional 
\begin{equation}
\nTAP^\beta(\mu):=
\inf_{\zeta\in \mathcal{M}_{0,1}} \nTAP^\beta(\mu,\zeta)
=\inf_{\zeta\in \mathcal{M}_{q,1}} \nTAP^\beta(\mu,\zeta).
\label{eqTAPfirst}
\end{equation}
We will denote the minimizer to the right-hand side by $\zeta_{\beta,\mu}.$ It was proved in \cite{GenTAP} that the minimizer is unique and that $\nTAP^\beta(\mu)$ is a continuous functional on $M_*$. Let us denote 
$$
S_N(q)= \Bigl\{m\in [-1,1]^N : \frac{1}{N}\|m\|^2=q\Bigr\}.
$$
For $m\in[-1,1]^N$, define the empirical measure 
\begin{equation}
\mu_m = \frac1N \sum_{i\leq N}\delta_{m_i}.
\end{equation}
The following were the main results in \cite{GenTAP}.
\begin{thm}[TAP correction]
\label{thm:TAPcorrection}
For any $c, t>0$, if $\eps,\delta>0$ are small enough and $n\geq 1$ is large enough then, for large $N$,
\begin{equation}
\label{eq:TAPunifconvbeta}
\p \Bigl( \forall m \in [-1,1]^N : \ \big|\nTAP^\beta_{N,n}(m,\eps,\delta)  -  \nTAP^\beta(\mu_m) \big| < t \Bigr)>1-e^{-cN}.
\end{equation}
\end{thm}

\begin{thm}[TAP representation]\label{thm:GenTAP}
For any $q\in \supp\, \zeta_{\beta}^*$  and any $t>0$,
\begin{equation}
\label{eq:GenTAPbeta}
\lim_{N\to\infty}\p\Bigl(\,
\Bigl|F_N(\beta) - \max_{m\in S_N(q)}\Bigl( \frac{H_N(m)}{N} + \frac{1}{\beta}\nTAP^\beta(\mu_m)\Bigr)\Bigr|
< t\,
\Bigr) = 1.
\end{equation}
\end{thm}

\begin{thm}[TAP states are ancestral]\label{Thm1label}
For any $q\in \supp\, \zeta_{\beta}^*$ and any $t>0$,
\begin{equation}
\label{eqTAPAS0}
\lim_{N\to\infty}\p\Bigl(\,
\Bigl|F_N(\beta) - \max_{m\in S_N(q)}\Bigl( \frac{H_N(m)}{N} + \frac{1}{\beta}\nTAP^\beta(\mu_m,\zeta_{\beta}^*)\Bigr)\Bigr|
< t\,
\Bigr) = 1.
\end{equation}
\end{thm}
\begin{thm}[Generalized TAP equations]\label{ThmGTElabbeta}
	For any $m\in (-1,1)^N\cap S_N(q),$
	\begin{equation}
	\nabla \nTAP^\beta(\mu_m) = -\frac{1}{N}\Bigl(\oPsi_\beta(q,m_i,\zeta_{\beta,m})+m_i\beta^2\xi''(q)\int_q^1\zeta_{\beta,m} ds\Bigr)_{i\leq N},
	\label{eqTAPstates12beta}
	\end{equation}
	where $\zeta_{\beta,m}:=\zeta_{\beta,\mu_m}$ is the minimizer to \eqref{eqTAPfirst} with $\mu=\mu_m.$
\end{thm}

\subsection{The energy of generalized TAP states}

Our first main result computes the energy and TAP correction for all generalized TAP states at positive temperature in terms of their distance from the origin. Given $\beta>0$, let us denote the TAP free energy functional by
\begin{equation}
f_{m}(\beta):=\frac{\beta H_N(m)}{N}+\nTAP^\beta(\mu_m),\,\,\forall m\in [-1,1]^N,
\end{equation}
and, given $\eps>0$, let
\begin{equation}
M_{\beta,q}(\eps):=\Bigl\{m\in S_N(q) : f_{m}(\beta)\geq \max_{m\in S_N(q)}f_m(\beta)-\eps\Bigr\}
\end{equation}
be the set of $\eps$-maximizers of $f_m(\beta).$ For simplicity of notation, we keep the dependence of $f_{m}(\beta)$ and $M_{\beta,q}(\eps)$ on $N$ implicit. The elements of the set $M_{\beta,q}(\eps_N)$ with $\eps_N \to 0$ and $q\in\supp\, \zeta_{\beta}^*$ are called the generalized TAP states.

 \begin{thm}[The energy of generalized TAP states] \label{zt:thm2}
 	For any $q\in \supp\,\zeta_{\beta}^*$ and any sequence $\eps_N\geq 0$ going to zero, almost surely,
 	\begin{align}
 	\begin{split}
 	\label{zt:lem:eq1}
 	\lim_{N\to\infty} \max_{M_{\beta,q}(\eps_N)}\Bigl|\frac{H_N(m)}{N}-E_\beta(q)\Bigr|=0
 	\end{split}
 	\end{align}
 	and
 	\begin{align}
 	\begin{split}\label{zt:lem:eq2}
 	\lim_{N\to\infty} \max_{M_{\beta,q}(\eps_N)}\Bigl|\nTAP^\beta(\mu_m)-\bigl(\mathcal{P}_\beta(\zeta_{\beta}^*)-\beta E_\beta(q)\bigr)\Bigr|=0,
 	\end{split}
 	\end{align}
 	where 
 	\begin{align}\label{zt:lem:eq3}
 	E_\beta(q):=\beta\int_0^q \xi''(s)\Bigl(\int_s^1\zeta_\beta^*(t)dt\Bigr)ds=\beta\xi'(q)\int_q^1\zeta_{\beta}^*(s)\,ds+\beta\int_0^q\!\xi'(s)\zeta_{\beta}^*(s)\,ds.
 	\end{align}
 
 \end{thm}
\begin{remark}[Classical case]\rm 
By definition, $\nTAP^\beta(\mu_m)\leq \nTAP^\beta(\mu_m,\zeta_{\beta}^*)$ and therefore, by Theorems \ref{thm:GenTAP} and \ref{Thm1label}, we must have $\nTAP^\beta(\mu_m)\approx \nTAP^\beta(\mu_m,\zeta_{\beta}^*)$ for all the generalized TAP states.
Classical TAP states correspond to $q=\qea$, which is the largest point in the support of $\zeta_{\beta}^*,$ in which case the TAP correction simplifies to (see \cite[Proposition 11]{GenTAP})
	\begin{align}\label{add:eq11}
		\nTAP^\beta(\mu_m)=-\frac{1}{N}\sum_{i=1}^{N} I(m_i)+\beta^2 C(\qea),
	\end{align}
	where 
	\begin{align*}
	I(a)&=\frac{1+a}{2}\log \frac{1+a}{2}+\frac{1-a}{2}\log \frac{1-a}{2},
	\\
	C(q)&=\frac{1}{2}\bigl(\xi(1)-\xi(q)-\xi'(q)(1-q)\bigr).
	\end{align*}
In particular, (\ref{zt:lem:eq2}) implies that the entropy of the classical TAP states is given by
\begin{equation}
-\frac{1}{N}\sum_{i=1}^{N} I(m_i) \approx
\PP_\beta(\zeta_{\beta}^*)-\beta\xi'(q_{\EA})(1-q_{\EA})-\beta^2\int_{0}^{q_{\mathrm{EA}} }\!\xi'(s)\zeta_{\beta}^*(s)\,ds
-\beta^2 C(\qea),
\end{equation}
so both energy and entropy of classical TAP states are constant.
\qed
\end{remark}

There exists an asymptotic description of measures $\mu_m$ corresponding to ancestor states $m$ in the Parisi ansatz, and we will derive an asymptotic analogue of (\ref{zt:lem:eq2}) directly from this description. Such description first appeared in the physics literature in \cite{M3}. Rigorously, an asymptotic distribution of spins (from which a description of $\mu_m$ can be extracted) in terms of the Parisi measure was derived in Chapter 4 in \cite{SKmodel} under certain regularizing perturbations that were introduced in \cite{Pspins}, and it was observed in \cite{AJ16} that for generic models the same proof works without perturbations. The results in \cite{SKmodel} were written in terms of the discrete Ruelle probability cascades, whose overlap distribution approximates the Parisi measure $\zeta_{\beta}^*$, but one can write them directly in terms of the Parisi measure (without discretization) in terms of the solution of the SDE 
\begin{equation}
dX(s)=\beta\xi''(s)\partial_x\Phi_{\zeta_{\beta}^*}^\beta(s,X(s))ds+\xi''(s)^{1/2}dW_s , X(0)=0,
\end{equation}
as was done, for example, in \cite{AJ16} and \cite{AufJag18}. We will not describe all these results precisely here, but simply mention that, for $q\in \supp\, \zeta_{\beta}^*$, asymptotically the coordinates of an ancestor state $m$ with $\frac{1}{N}\|m\|^2=q$ look like i.i.d. random variables with the distribution  
\begin{equation}
	\mu_q(\,\cdot\,)=\p\Bigl(\partial_x\Phi_{\zeta_\beta^*}^\beta\bigl(q,X(q)\bigr)\in \,\cdot\,\Bigr).
\label{eqmuqdef}
\end{equation}
In other words, $\mu_q$ is an asymptotic analogue of $\mu_m$. We will show the following.
\begin{thm}\label{add:thm1} 
For any $q\in \supp\, \zeta_{\beta}^*$ and $\mu_q$ defined in (\ref{eqmuqdef}),
	\begin{align}
		\nTAP^\beta(\mu_q)=\mathcal{P}_\beta(\zeta_{\beta}^*)-\beta E_\beta(q).
	\end{align}
Moreover, for any $q\in [0,1),$
	\begin{align}
	\nTAP^\beta(\mu_q)\leq \mathcal{P}_\beta(\zeta_{\beta}^*)-\beta E_\beta(q).
	\end{align}
\end{thm}
The first equation is an asymptotic analogue of (\ref{zt:lem:eq2}), and the second equation states that, in general, $\mathcal{P}_\beta(\zeta_{\beta}^*)-\beta E_\beta(q)$ is an upper bound on the TAP correction for such measures.

\subsection{TAP approach at zero temperature}
Next, we will describe the analogue of the above results at zero temperature. Let us define 
\begin{align}
\nTAP_{N,1}^\infty(m,\varepsilon)
&=\frac{1}{N}\max_{B(m,\varepsilon)}\bigl(H_N(\bs)-H_N(m)\bigr),\\
\nTAP_{N,n}^\infty(m,\varepsilon,\delta)&=\frac{1}{nN}\max_{B_n(m,\varepsilon,\delta)}\sum_{i=1}^n\bigl(H_N(\bs^i)-H_N(m)\bigr).\label{eq:TAPn}
\end{align}
Then we can write
\begin{equation}
\label{eq:basic_ineq}
\max_{\bs\in\Sigma_N} \frac{H_N(\bs)}{N}
\geq \frac{H_N(m)}{N} + \nTAP^\infty_{N,1}(m,\eps)
\geq \frac{H_N(m)}{N} + \nTAP^\infty_{N,n}(m,\eps,\delta).
\end{equation}
We will be interested in points $m\in[-1,1]^N$ where the above inequalities become approximate equalities, for large $N$. In other words, we are interested to characterize points $m$ that have many near ground states orthogonal to each other relative to $m$.

Let $\oN$ be the family of c.d.f.s induced by all measures $\gamma$ on $[0,1)$ with 
\begin{equation}
\int_0^1\!\gamma(s)\,ds=\int_0^1\!\gamma([0,s])\,ds<\infty.
\label{eqGammaBound}
\end{equation}
For $\gamma\in \oN$, consider the solution $\Theta_{\gamma}$ to the following PDE,
\begin{align}\label{zero:eq2}
\partial_t\Theta_{\gamma}=-\frac{\xi''(t)}{2}\Bigl(\partial_{xx}\Theta_{\gamma}+\gamma(t) \bigl(\partial_x\Theta_{\gamma}\bigr)^2\Bigr)
\end{align}
on $[0,1]\times\mathbb{R}$ with the boundary condition $\Theta_\gamma(1,x)=|x|.$ It was shown in \cite[Corollary 2]{AC15} (see also \cite[Section 2]{CHL}) how such solution $\Theta_{\gamma}(t,x)$ can be defined for all $(t,x)\in [0,1]\times\mathbb{R}$ under the condition (\ref{eqGammaBound}). For $ a\in [-1,1]$, we define
\begin{align}
\Lambda_\gamma^\infty(q,a):=\inf_{x\in \mathbb{R}}\Bigl(\Theta_{\gamma}(q,x)-ax\Bigr).
\label{eqLambdaDef}
\end{align}
We will see that, for $a\in (-1,1)$, the minimizer is unique and finite (see Remark \ref{lalRemarkAL1} below). We will denote this minimizer by $\oPsi(q,a,\gamma)$, so that
\begin{equation}
\Lambda^\infty_\gamma(q,a):=\Theta_\gamma\bigl(q,\oPsi(q,a,\gamma)\bigr)-a\oPsi(q,a,\gamma), \,\, a\in (-1,1).
\label{eqCCParPsi}
\end{equation}
Moreover, for $a=\pm 1$, this infimum is well-defined and (see Remark \ref{lalRemarkPM1} below)
\begin{equation}
\Lambda_{\gamma}^\infty (q,\pm 1)=\frac{1}{2}\int_q^1\!\xi''(s)\gamma(s)\,ds.
\end{equation}
If $\mu\in M_*$ with $q=\int\! a^2\,d\mu(a),$ we define
\begin{align}
\nTAP^\infty(\mu,\gamma)=\int\! \Lambda_\gamma^\infty(q,a)\,d\mu(a)-\frac{1}{2}\int_q^1\! s\xi''(s)\gamma(s)\,ds.
\end{align}
Again, notice that this functional depends  only on the values of $\gamma(s)$ on the interval $[q,1]$, so we can view it as a functional on the space $\nN$ of measures on $[q,1)$ such that 
$$
\int_q^1\! \gamma(s)\,ds=\int_q^1\! \gamma([q,s])\,ds<\infty.
$$ 
Finally, we let
\begin{align}
\label{eqTAPinfty}
\nTAP^\infty(\mu)
=\inf_{\gamma\in \oN}\nTAP^\infty(\mu,\gamma)
=\inf_{\gamma\in \nN}\nTAP^\infty(\mu,\gamma).
\end{align}
We are now ready to state our main results on the generalized TAP free energy at zero temperature. The first is a uniform concentration result for the TAP free energy defined in \eqref{eq:TAPn} around the (non-random) functional we have just defined \eqref{eqTAPinfty}, applied to the empirical measure 
$\mu_m = \frac1N \sum_{i\leq N}\delta_{m_i}$.
\begin{thm}[TAP correction at zero temperature]\label{ThmZTut}
For any $c,t>0$, if $\varepsilon,\delta>0$ are small enough and $n$ is large enough then, for large $N,$
\begin{align}
\p\Bigl(\forall m\in [-1,1]^N:\bigl|\nTAP_{N,n}^\infty(m,\varepsilon,\delta)-\nTAP^\infty(\mu_m)\bigr|<t\Bigr)>1-e^{-cN}.
\label{eqThmZTut}
\end{align}	
\end{thm}

Recall that the Parisi formula for the ground state energy of the mixed $p$-spin model derived in \cite{AC17} states that
\begin{align}\label{GSE}
	\lim_{N\rightarrow\infty}\max_{\bs\in \Sigma_N}\frac{H_N(\bs)}{N}=\inf_{\gamma\in \mathcal{N}_{0,1}}\Bigl(\Theta_{\gamma}(0,0)-\frac{1}{2}\int_0^1\!s\xi''(s)\gamma(s)\,ds\Bigr),
\end{align}
and this variational formula has a unique minimizer, denoted $\gamma^*.$ The next result is the TAP representation for the ground state energy, which is the zero-temperature
 analogue of the TAP representation for the free energy in Theorem \ref{thm:GenTAP} above.
\begin{thm}[TAP representation at zero temperature]\label{ThmTAPrepZT}
	For any $q\in\supp\,\gamma^*$ and any $t>0,$
	\begin{equation}
	\label{eq:GenTAP}
	\lim_{N\to\infty}\p\Bigl(\,
	\Bigl|\max_{\bs\in \Sigma_N}\frac{H_N(\bs)}{N} - \max_{m\in S_N(q)}\Bigl( \frac{H_N(m)}{N} + \nTAP^\infty(\mu_m)\Bigr)\Bigr|
	< t\,
	\Bigr) = 1.
	\end{equation}
\end{thm}

Note that by combining the two theorems above, if $m$ is an approximate maximizer in \eqref{eq:GenTAP}, then the inequalities of \eqref{eq:basic_ineq} become approximate equalities. Namely,
\begin{equation}
\max_{\bs\in\Sigma_N} \frac{H_N(\bs)}{N}
\approx \frac{H_N(m)}{N} + \nTAP^\infty_{N,1}(m,\eps)
\approx \frac{H_N(m)}{N} + \nTAP^\infty_{N,n}(m,\eps,\delta),
\end{equation}
provided that $\eps$ and $\delta$ are small enough, and $n$ is large enough.
In other words, any generalized TAP state contains many samples $\bs^i\in B(m,\eps)$ which approximately maximize the energy, and such that the centered samples $\tbs^i=\bs^i-m$ are approximately orthogonal.

Recall that the functional $\nTAP^\infty(\mu)$ was defined in \eqref{eqTAPinfty} as an infimum over the space of c.d.f.s $\nN$. The following theorem shows that the minimizer is unique. 
\begin{thm}\label{zero:thm4}
For any $\mu\in M_*$ with $q=\int\! a^2\,d\mu(a),$ $\gamma\to\nTAP^\infty(\mu,\gamma)$ has a unique minimizer $\gamma_{\mu}\in\nN$.
\end{thm}

We think of the minimizer as the order parameter associated to a generalized TAP state with $\mu_m=\mu$. It is related the order parameter of the original model through the following theorem.
\begin{thm}[Ancestral property of zero-temperature TAP states]\label{ThmAncZT}
	For any $q\in\supp\,\gamma^*$ and any $t>0$,
	\begin{equation}
	\label{eqTAPAS}
	\lim_{N\to\infty}\p\Bigl(\,\Bigl|\max_{\bs\in \Sigma_N}\frac{H_N(\bs)}{N}- \max_{m\in S_N(q)}\Bigl( \frac{H_N(m)}{N} + \nTAP^\infty(\mu_m,\gamma^*)\Bigr)\Bigr|<t\,\Bigr)=1.
	\end{equation}
\end{thm}
Note that if $m$ is an approximate maximizer in \eqref{eq:GenTAP}, then it must also be an approximate maximizer  of \eqref{eqTAPAS} and
\begin{equation}
	\nTAP^\infty(\mu_m) \approx \nTAP^\infty(\mu_m,\gamma^*).
\end{equation}

Next, in order to describe the critical point equations for the TAP states,
\begin{equation}
\frac{1}{N}\nabla H_N(m)=-\nabla \nTAP^\infty(\mu_m),
\label{eqTAPstates1}
\end{equation} 
we need to compute the gradient of $\nTAP^\infty(\mu_m).$ The statement is somewhat more involved than what one would expect from the direct analogue of Theorem \ref{ThmGTElabbeta} above. Denote $\gamma_m:=\gamma_{\mu_m}$ and let
\begin{equation}
\Delta(m):= \frac{1}{\xi'(1)-\xi'(q)}\Bigl(\nTAP^\infty(\mu_m)-\int_q^1\! (\xi'(s)-\xi'(q))\gamma_m(s)\,ds\Bigr).
\label{eqDefDelta}
\end{equation}
\begin{thm}[Gradient of TAP correction]\label{ThmGTElab}
For any $m\in (-1,1)^N$ with $\frac{1}{N}\|m\|^2=q,$ if we denote
\begin{equation}
C(m):= \xi''(q)\! \int_{q}^1\gamma_m(s)\,ds  + \xi''(q) \Delta(m),
\end{equation}
then
\begin{align}
\nabla \nTAP^\infty(\mu_m) 
&= -\frac{1}{N}\Bigl(\oPsi(q,m_i,\gamma_{m})+C(m)m_i\Bigr)_{i\leq N}.
\label{eqTAPstates12}
\end{align}
\end{thm}

If we combine (\ref{eqTAPstates1}) and (\ref{eqTAPstates12}), we can write
$$
(\nabla H_N(m))_i-C(m)m_i=\oPsi(q,m_i,\gamma_{m}).
$$
If we plug both sides into $\partial_x\Theta_{\gamma_{m}}(q,\,\cdot\,)$ and recall the definition of $\oPsi$, we get
\begin{equation}
\partial_x\Theta_{\gamma_{m}}\Bigl(q,(\nabla H_N(m))_i-C(m)m_i\Bigr)
= m_i.
\label{exactTAP}
\end{equation}
These are the TAP equations at zero temperature.

\section{Passing to zero temperature}

Some of the zero temperature results above can be proved by adapting the proofs from \cite{GenTAP} to the zero-temperature setting. This, however, entails a rather involved and long analysis. Instead, the approach we shall take here is to relate the zero-temperature variants to the results proved  for positive temperature in \cite{GenTAP}, and use those as much as possible.
The main result of this section is Lemma \ref{zero:lem1} below, that bounds,
for a given empirical measure $\mu$, the difference between the functional $\nTAP^\beta(\mu)$ (see \eqref{eqTAPfirst}) at a given positive temperature and the zero-temperature functional $\nTAP^\infty(\mu)$ (see \eqref{eqTAPinfty}).
It will allow us to reduce zero-temperature results to the positive temperature results in the previous section.
We first prove the following simple consequence of Theorem \ref{thm:TAPcorrection}, that bounds  the 
difference of the functional $\nTAP^\beta(\mu)$ at two different temperatures.
\begin{lem}
For any $0<\beta_1\leq \beta_2$ and $\mu\in M_*$,
\begin{equation}\label{eqTAPBcomp}
\Bigl|\frac{1}{\beta_1}\nTAP^{\beta_1}(\mu)-\frac{1}{\beta_2}\nTAP^{\beta_2}(\mu)\Bigr|\leq \frac{\log 2}{\beta_1}.
\end{equation}
\end{lem}

\begin{proof}
For a fixed $t>0$, by (\ref{eq:TAPunifconvbeta}) and Gaussian concentration,
$$
\big|\e \nTAP^{\beta_j}_{N,n}(m,\eps,\delta)  -  \nTAP^{\beta_j}(\mu_m) \big| < 2 t
$$
for $j=1,2,$ for large enough $N$. On the other hand,
$$
\e \nTAP^{\infty}_{N,n}(m,\eps,\delta) 
\leq
\frac{1}{\beta_j}\e \nTAP^{\beta_j}_{N,n}(m,\eps,\delta) 
\leq 
\e \nTAP^{\infty}_{N,n}(m,\eps,\delta) +\frac{\log 2}{\beta_j}
$$
and, therefore,
$$
\Bigl| \frac{1}{\beta_1}\e \nTAP^{\beta_1}_{N,n}(m,\eps,\delta)  
- \frac{1}{\beta_2}\e \nTAP^{\beta_2}_{N,n}(m,\eps,\delta) \Bigr| 
\leq \frac{\log 2}{\beta_1}.
$$
This implies that
$$
\Bigl| \frac{1}{\beta_1}\nTAP^{\beta_1}(\mu_m) - \frac{1}{\beta_2}\nTAP^{\beta_2}(\mu_m)\Bigr| 
\leq  \frac{\log 2+4t}{\beta_1}.
$$
Choosing $m=m^N$ so that $\mu_m\to\mu$ and using continuity of $\nTAP^\beta$ proves the same inequality for arbitrary $\mu\in M_*$. Since $t$ is arbitrary, we get (\ref{eqTAPBcomp}).
\end{proof}

Let us denote an $L^1$-distance on $\nN$ by
$$
d_1(\gamma,\gamma'):=\int_q^1|\gamma(s)-\gamma'(s)|\,ds.
$$ 
It was proved in \cite[Corollary 2]{AC15} and \cite[Proposition 2]{CHL} that 
\begin{equation}
\Bigl|\Theta_\gamma(t,x)- \Theta_{\gamma'}(t,x)\Bigr|
\leq 2\xi''(1)d_1(\gamma,\gamma').
\label{eqLipInf}
\end{equation}
Since
\begin{equation}
\Bigl| \Lambda_\gamma^\infty(q,a)-\Lambda_{\gamma'}^\infty(q,a)\Bigr|
\leq
\sup_x \Bigl|\Theta_\gamma(t,x)- \Theta_{\gamma'}(t,x)\Bigr|
\leq 2\xi''(1)d_1(\gamma,\gamma'),
\label{eqLipLamInf}
\end{equation}
we get that
\begin{equation}
\Bigl|\nTAP^\infty(\mu,\gamma)-\nTAP^\infty(\mu,\gamma')\Bigr|
\leq 3\xi''(1)d_1(\gamma,\gamma').
\label{eqLipInfTAP}
\end{equation}
Hence, $\gamma\to\nTAP^\infty(\mu,\gamma)$ is Lipschitz on $(\nN,d_1)$, which will be useful in the proof of our next result.
\begin{lem}\label{zero:lem1}
	For any $\beta>0$ and $\mu\in M_*,$ we have that
	\begin{align}
    \Bigl|\nTAP^\infty(\mu)-\frac{1}{\beta}\nTAP^\beta(\mu)\Bigr|\leq \frac{\log 2}{\beta}.
    \label{zero:eq4}
	\end{align}
	\end{lem}

\begin{proof}
With $q=\int\! a^2\,d\mu(a)$, let $\zeta \in\mathcal{M}_{q,1}$. If we make the change of variables
	\begin{align*}
	\Theta_{\beta\zeta}^\beta(t,x):=\frac{1}{\beta}\Phi_{\zeta}^\beta(t,\beta x),
	\end{align*}
it is easy to check that 
\begin{align*}
\partial_t\Theta_{\beta\zeta}^\beta=-\frac{\xi''(t)}{2}\Bigl(\partial_{xx}\Theta_{\beta\zeta}^\beta+\beta \zeta(t)\bigl(\partial_x\Theta_{\beta\zeta}^\beta\bigr)^2\Bigr)
\end{align*}
with the boundary condition 
$$
\Theta_{\beta\zeta}^\beta(1,x)=\frac{1}{\beta}\log 2\cosh (\beta x).
$$
Standard properties of the Parisi functional $\Phi_{\zeta}^\beta$ extend to $\Theta_{\zeta}^\beta$.  For example, it is well-known that changing the boundary condition in the definition of $\Phi_{\zeta}^\beta$ by at most a constant changes the solution by at most this constant, so the same holds  for $\Theta_{\zeta}^\beta$. Observe that
\begin{align*}
|x|\leq \frac{1}{\beta}\log 2\cosh (\beta x)\leq |x|+\frac{\log 2}{\beta}.
\end{align*}
Since $\Theta_{\beta\zeta}^\beta$ and $\Theta_{\beta\zeta}$ in (\ref{zero:eq2}) only differ in the boundary conditions, which differ by at most $\log2/\beta,$ we get
\begin{align}\label{zt:add:eq1}
\Theta_{\beta\zeta}(q,x)\leq \Theta_{\beta\zeta}^\beta(q,x)\leq \Theta_{\beta\zeta}(q,x)+\frac{\log 2}{\beta}.
\end{align}
Using this together with
\begin{align*}
\inf_{x\in \mathbb{R}}\Bigl(\Theta_{\beta\zeta}^\beta(q,x)-ax\Bigr)
&=\inf_{x\in \mathbb{R}}\Bigl(\frac{1}{\beta}\Phi_{\zeta}^\beta(q,\beta x)-ax\Bigr)
\\
&=\frac{1}{\beta}\inf_{x\in \mathbb{R}}\Bigl(\Phi_{\zeta}^\beta(q,x)-ax\Bigr)
=
\frac{1}{\beta}\Lambda_{\zeta}^\beta (q,a)
\end{align*}
 implies that
\begin{align}
\Lambda^\infty_{\beta\zeta} (q,a)\leq 
\frac{1}{\beta} \Lambda_{\zeta}^\beta (q,a)
\leq \Lambda_{\beta\zeta}^\infty (q,a)+\frac{\log 2}{\beta}.
\label{eqLambdaRel}
\end{align}
Note also that 
\begin{align*}
\frac{1}{\beta}\frac{\beta^2}{2}\int_q^1\! s\xi''(s)\zeta(s)\,ds=\frac{1}{2}\int_q^1\! s\xi''(s)\bigl(\beta\zeta(s)\bigr)\,ds.
\end{align*}
Combining the last two displays, we get
\begin{align}\label{zero:eq1}
\nTAP^{\infty}(\mu,\beta\zeta)
\leq
\frac{1}{\beta}\nTAP^{\beta}(\mu,\zeta)
\leq
\nTAP^{\infty}(\mu,\beta\zeta) + \frac{\log 2}{\beta}.
\end{align}
If we denote by $\nN^\beta$ the set of all measures on $[q,1]$ of total mass at most $\beta$, taking infimum over all $\zeta\in \mathcal{M}_{q,1}$ gives
\begin{align*}
\inf_{\gamma\in \nN^\beta}\nTAP^{\infty}(\mu,\gamma)
\leq
\frac{1}{\beta}\nTAP^{\beta}(\mu)
\leq 
\inf_{\gamma\in \nN^\beta}\nTAP^{\infty}(\mu,\gamma)+ \frac{\log 2}{\beta}.
\end{align*}
As $\beta\uparrow\infty,$ by (\ref{eqLipInfTAP}), the infimum over $\gamma\in \nN^\beta$ converges to the infimum over all $\gamma\in \nN,$ and using (\ref{eqTAPBcomp}) finishes the proof.
\end{proof}

\begin{remark}\label{lalRemarkPM1}\rm
It was shown in the proof of Theorem 10 $(i)$ in \cite{GenTAP} that
$$
\Lambda_{\zeta}^\beta (q,\pm 1)=\frac{\beta^2}{2}\int_q^1\!\xi''(s)\zeta(s)\,ds,
$$
which together with (\ref{eqLambdaRel}) implies that
$$
\Bigl|
\Lambda_{\beta\zeta}^\infty (q,\pm 1)-\frac{1}{2}\int_q^1\!\xi''(s)(\beta\zeta(s))\,ds
\Bigr|\leq \frac{\log 2}{\beta}.
$$
By (\ref{eqLipLamInf}), it follows that
\begin{equation}
\Lambda_{\gamma}^\infty (q,\pm 1)=\frac{1}{2}\int_q^1\!\xi''(s)\gamma(s)\,ds
\end{equation}
for all $\gamma\in \nN.$
\qed
\end{remark}

\section{TAP correction and representation}

In this section we combine Lemma \ref{zero:lem1} from the previous section and Theorems \ref{thm:TAPcorrection} and \ref{thm:GenTAP}, which concern the positive temperature case, to prove their zero-temperature analogues, 
Theorems \ref{ThmZTut} and \ref{ThmTAPrepZT}.

\subsection{Proof of Theorem \ref{ThmZTut}}
Note that
$$
\nTAP^{\infty}_{N,n}(m,\eps,\delta) 
\leq
\frac{1}{\beta} \nTAP^{\beta}_{N,n}(m,\eps,\delta) 
\leq 
 \nTAP^{\infty}_{N,n}(m,\eps,\delta) +\frac{\log 2}{\beta}.
$$
Together with Lemma \ref{zero:lem1} this implies that
\begin{align*}
\bigl|\nTAP_{N,n}^\infty(m,\varepsilon,\delta)-\nTAP^{\infty}(\mu_m)\bigr|
&\leq 
\frac{1}{\beta}\Bigl|\nTAP_{N,n}^\beta(m,\varepsilon,\delta)-\nTAP^{\beta}(\mu_m)\Bigr|+\frac{2\log 2}{\beta}.
\end{align*}
This implies that the probability on the left-hand side of (\ref{eqThmZTut}) is bounded from below by
$$
\p\Bigl(\forall m\in [-1,1]^N:\bigl|\nTAP_{N,n}^\beta(m,\varepsilon,\delta)-\nTAP^{\beta}(\mu_m)\bigr|<\beta t-2\log 2\Bigr).
$$
If we take $\beta$ large enough so that $\beta t>2\log 2$ then our claim follows from Theorem \ref{thm:TAPcorrection}.\qed

\subsection{Proof of Theorem \ref{ThmTAPrepZT}}

Recall that it was proved in \cite{AC15} that if we denote by $\zeta_{\beta}^*\in \mathcal{M}_{0,1}$ the minimizer of the Parisi formula for the free energy $F_N(\beta)$ of the original model, then $\beta \zeta_{\beta}^*\rightarrow \gamma^*$ vaguely, i.e., for any continuous function $f$ with compact support in $[0,1)$, 
\begin{align*}
\lim_{\beta\rightarrow\infty}\int_{[0,1)}\! f(s)\,d\beta\zeta_{\beta}^*(s)=\int_{[0,1)}\! f(s)\,d\gamma^*(s).
\end{align*} 
Fix $q$ in the support of $\gamma^*$. Then there exists $q_\beta$ in the support of $\zeta_{\beta}^*$ such that $q_\beta\to q$ as $\beta\to \infty.$ Note that
\begin{align*}
\Bigl|F_N(\beta)-\max_{\bs\in \Sigma_N}\frac{H_N(\bs)}{N}\Bigr|\leq \frac{\log 2}{\beta}
\end{align*}
and, from \eqref{zero:eq4},
\begin{align*}
\Bigl|\frac{1}{\beta}\nTAP^\beta(\mu)-\nTAP^\infty(\mu)\Bigr|\leq \frac{\log 2}{\beta}.
\end{align*}
From these,
\begin{align}
\begin{split}\label{zero:ThmTAPrepZT:proof:eq4}
&\Bigl|\Bigl(F_N(\beta)-\max_{m\in S_N(q_\beta)}\Bigl(\frac{H_N(m)}{N}+\frac{1}{\beta}\nTAP^\beta(\mu_m)\Bigr)\Bigr)\\
&\qquad-\Bigl(\max_{\bs\in \Sigma_N}\frac{H_N(\bs)}{N}-\max_{m\in S_N(q_\beta)}\Bigl(\frac{H_N(m)}{N}+\nTAP^\infty(\mu_m)\Bigr)\Bigr)\Bigr|\leq \frac{2\log 2}{\beta}.
\end{split}
\end{align}
To handle the second big bracket, observe that since $\nTAP^\beta$ is continuous, it follows from \eqref{zero:eq4} that $\nTAP^\infty $ is uniformly continuous on $M_*$, since $M_*$ is compact. Hence, for any $\varepsilon>0,$ there exists $0<\delta<\min(\varepsilon,(1-q)/2)$ such that $|\nTAP^\infty(\mu)-\nTAP^\infty(\mu')|<\varepsilon$ whenever $\mu,\mu'\in M_*$ satisfy $d_1(\mu,\mu')\leq \delta.$ From now on, we fix $\beta$ large enough so that $q_\beta\in [q-\delta,q+\delta].$
Note that for any $m\in [-1,1]^N$ with $\|m\|^2/N=q$, we can find $ m' $ with $\|m'\|^2/N=q_\beta$ such that $\|m- m' \|\leq \delta\sqrt{N}.$ Furthermore, we can choose $ m' $ so that the absolute values of the coordinates of $m$ and $m'$ are arranged in the same order,
$$
|m_{i_1}|\leq |m_{i_2}|\leq \cdots\leq |m_{i_N}|\Longrightarrow
|m_{i_1}'|\leq |m_{i_2}'|\leq \cdots\leq |m_{i_N}'|.
$$ 
If $\mu_{|m|}:=\frac{1}{N}\sum_{i\leq N}\delta_{|m_i|}$ then
\begin{align*}
d_1(\mu_{|m|},\mu_{| m' |})=\frac{1}{N}\sum_{i=1}^N\bigl||m_i|-|m_{i}' |\bigr|\leq \frac{1}{N}\sum_{i=1}^N\bigl|m_i-m_{i}' \bigr|\leq \frac{\|m- m' \|}{\sqrt{N}}\leq \delta.
\end{align*}
Hence, from the above uniform continuity,
\begin{align}\label{zero:ThmTAPrepZT:proof:eq2}
\bigl|\nTAP^\infty(\mu_{m})-\nTAP^\infty(\mu_{ m' })\bigr|=\bigl|\nTAP^\infty(\mu_{|m|})-\nTAP^\infty(\mu_{| m' |})\bigr|<\varepsilon.
\end{align}
In a similar manner, for any $m\in [-1,1]^N$ with $\|m\|^2/N=q_\beta$, we can find $m' \in [-1,1]^N$ with $\|m' \|^2/N=q$ so that $\|m-m' \|\leq \delta\sqrt{N}$ and
\begin{align}\label{zero:ThmTAPrepZT:proof:eq3}
\bigl|\nTAP^\infty(\mu_{m})-\nTAP^\infty(\mu_{m' })\bigr|=\bigl|\nTAP^\infty(\mu_{|m|})-\nTAP^\infty(\mu_{|m' |})\bigr|<\varepsilon.
\end{align}
On the other hand, from the Dudley entropy integral formula, there exists a constant $C>0$ depending only on $\xi$ such that
\begin{align*}
\e\max_{\|m-m'\|<\delta\sqrt{N}}\Bigl|\frac{H_N(m)}{N}-\frac{H_N(m')}{N}\Bigr|\leq C\delta,
\end{align*}
which, combined with the Gaussian concentration inequality, implies that
\begin{align*}
\max_{\|m-m'\|<\delta\sqrt{N}}\Bigl|\frac{H_N(m)}{N}-\frac{H_N(m')}{N}\Bigr|\leq 2C\delta
\end{align*}
with probability at least $1-C'e^{-C'\delta N}$, where $C'$ is a constant depending only on $\xi.$ Hence, from this inequality and \eqref{zero:ThmTAPrepZT:proof:eq2},
\begin{align*}
\Bigl|\max_{m\in S_N(q)}\Bigl(\frac{H_N(m)}{N}+\nTAP^\infty(\mu_m)\Bigr)-\max_{m\in S_N(q_\beta)}\Bigl(\frac{H_N(m)}{N}+\nTAP^\infty(\mu_{m})\Bigr)\Bigr|\leq \varepsilon(1+2C)
\end{align*}
with probability at least $1-2C'e^{-C'\delta N}$. Thus, from \eqref{zero:ThmTAPrepZT:proof:eq4},
\begin{align*}
&\Bigl|\Bigl(F_N(\beta)-\max_{m\in S_N(q_\beta)}\Bigl(\frac{H_N(m)}{N}+\frac{1}{\beta}\nTAP^\beta(\mu_m)\Bigr)\Bigr)\\
&\qquad-\Bigl(\max_{\bs\in \Sigma_N}\frac{H_N(\bs)}{N}-\max_{m\in S_N(q)}\Bigl(\frac{H_N(m)}{N}+\nTAP^\infty(\mu_m)\Bigr)\Bigr)\Bigr|\leq \frac{2\log 2}{\beta}+\varepsilon(1+2C).
\end{align*}
Our result then follows by using Theorem \ref{thm:GenTAP}.  
\qed

\section{Ancestral property of TAP states}

This section is dedicated to the proof of Theorem \ref{ThmAncZT}. Unlike in the previous section, here we work at zero-temperature directly.
First, note that since $\nTAP^\infty(\mu_m)\leq \nTAP^\infty(\mu_m,\gamma^*)$, using Theorem \ref{ThmTAPrepZT} and Gaussian concentration, our proof will be complete if we can show that, whenever $q$ lies in the support of $\gamma^*,$
	\begin{align*}
		\limsup_{N\rightarrow\infty}\e \max_{m\in S_N(q)}\Bigl(\frac{H_N(m)}{N}+\nTAP^\infty(\mu_m,\gamma^*)\Bigr)&\leq \mathcal{P}^\infty(\gamma^*).
	\end{align*}
Let $\mathcal{N}_{0,q}$ be the space of all cumulative distribution functions $\gamma$ induced by positive measures on $[0,q]$ satisfying $\int_0^q\gamma(s)\,ds <\infty.$ From Guerra's RSB bound for the ground state energy,
	\begin{align*}
	&\lim_{N\rightarrow\infty}\e\max_{m\in S_N(q)}\Bigl(\frac{H_N(m)}{N}+\nTAP^\infty(\mu_m,\gamma^*)\Bigr)\\
	&= 	\lim_{N\rightarrow\infty}\e\max_{m\in S_N(q)}\Bigl(\frac{H_N(m)}{N}+\frac{1}{N}\sum_{i=1}^N\Lambda_{\gamma^*}^\infty (q,m_i)-\frac{1}{2}\int_q^1s\xi''(s)\gamma^*(s)\,ds \Bigr)\\
	&\leq \mathcal{P}_q(\lambda,\gamma)-\frac{1}{2}\int_q^1s\xi''(s)\gamma^*(s)\,ds 
	\end{align*}
	for any $\gamma\in \mathcal{N}_{0,q}$, where
	\begin{align*}
	\mathcal{P}_q(\lambda,\gamma):=\Theta_{\gamma}^\lambda(0,0)-\frac{1}{2}\int_0^qs\xi''(s)\gamma(s)\,ds ,
	\end{align*}
	and where $\Theta_{\gamma}^\lambda$ is the solution to
	\begin{align*}
	\partial_t\Theta_\gamma^\lambda=-\frac{\xi''(t)}{2}\Bigl(\partial_{xx}\Theta_\gamma^\lambda+\gamma(t)\bigl(\partial_x\Theta_\gamma^\lambda\bigr)^2\Bigr)
	\end{align*}
	on $[0,q]\times\mathbb{R}$ with the boundary condition 
	$$
	\Theta_\gamma^\lambda(q,x):=\max_{a\in [-1,1]}\Bigl(ax+\lambda(a-q)+\Lambda_{\gamma^*}^\infty(q,a)\Bigr).
	$$
	If now we take $\lambda=0$ and $\gamma=\gamma^*1_{[0,q]},$ then from the conjugation, 
	$
		\Theta_\gamma^\lambda(q,x)=\Theta_{\gamma^*}(q,x)
	$ and thus, $\Theta_{\gamma}^\lambda(0,0)=\Theta_{\gamma^*}(0,0).$
	As a consequence,
	\begin{align*}
	\mathcal{P}_q^\infty(\lambda,\gamma)-\frac{1}{2}\int_q^1s\xi''(s)\gamma^*(s)\,ds &=\mathcal{P}^\infty(\gamma^*).
	\end{align*}
	This finishes our proof.
\qed

\section{Continuity of the Parisi functional}

In this section we will prove that the Parisi functional is continuous when defined on an extension of $\nN$ to measures that charge the point $1$. Namely, we set $\bN_{q,1}$ to be the collection of all measures on $[0,1]$ of the form 
\begin{equation}\label{eq5.1}
\nu(A)=\int_{A}\!\gamma(s)\,ds+\Delta \delta_1(A),
\end{equation}
where $\gamma([0,q))=0$, $\gamma|_{[q,1)}\in \nN$, and $\Delta\in [0,\infty).$ We equip $\bN_{q,1}$ with the topology of vague convergence.

\begin{remark}\rm
	Note that if $\nu_n\in \cup_{q\in [0,1]}{\bN}_{q,1}$ converges vaguely to certain $\nu_0\in\cup_{q\in [0,1]}{\bN}_{q,1}$, then $\gamma_n$ converges to $\gamma_0$ a.e. on $[0,1)$, where $(\gamma_n,\Delta_n)$ for $n\geq 0$ are the pairs associated to $\nu_n$ on $[0,1).$ Indeed, this can be seen by noting that $\nu_n([0,\cdot])$ for $n\geq 0$ are convex functions on $[0,1)$ and that $\lim_{n\to\infty}\nu_n([0,s])=\nu_0([0,s])$ for all $s\in [0,1)$ due to the vague convergence of $\nu_n$ and the fact that $\nu_0([0,\cdot])$ is continuous on $[0,1).$ Since $\nu_n([0,\cdot])$ for $n\geq 0$ are almost surely differentiable, we see that at the points of simultaneous differentiability of $\nu_n([0,s])$ for $n\geq 0,$ the Griffith lemma (see, e.g., \cite{Talbook1}) implies
	$$
	\lim_{n\to\infty}\gamma_n(s)=\lim_{n\to\infty}\frac{d}{ds}\nu_0([0,s])=\frac{d}{ds}\nu_0([0,s])=\gamma_0(s).
	$$
	We also mention that it is not necessarily true that $\Delta_n\to\Delta_0,$ instead the following limit is valid $$
	\Delta_0=\lim_{n\to\infty}\nu_n([0,1])-\int_0^1\gamma_0(s)ds.
	$$
\end{remark}

For each $\nu\in \bN_{q,1}$, if $(\gamma,\Delta)$ is the pair associated with $\nu$, we define 
$$
\Theta_{\nu}(q,x):=\Theta_{\gamma}(q,x)+\frac{\Delta \xi''(1)}{2}.
$$
In addition, we also define, for $a\in [-1,1],$
\begin{align*}
\Lambda_{\nu}^\infty(q,a):=\inf_{x\in \mathbb{R}}\Bigl(\Theta_\nu(q,x)-ax\Bigr)
=\Lambda_\gamma^\infty(q,a)+\frac{\Delta\xi''(1)}{2}.
\end{align*}
The main result of this section is the following proposition, which establishes the continuity of $\Theta_{\nu}(q,\cdot)$. It will be used in the proof of Theorems \ref{zero:thm4} and \ref{ThmGTElab}.

\begin{prop}\label{add:prop1}
	For any $q\in [0,1]$, if $\nu_n\to \nu_0$ vaguely in $\bN_{q,1}$ then 
	\begin{align}\label{add:prop1:eq1}
	\lim_{n\rightarrow\infty}\sup_{x\in \mathbb{R}}\bigl|\Theta_{\nu_n}(q,x)-\Theta_{\nu_0}(q,x)\bigr|=0
	\end{align}
	and
	\begin{align}\label{add:prop1:eq2}
	\lim_{n\rightarrow\infty}\sup_{a\in [-1,1]}\bigl|\Lambda^\infty_{\nu_n}(q,a)-\Lambda^\infty_{\nu_0}(q,a)\bigr|=0.
	\end{align}
\end{prop}
We also prove the following corollary, which will be used in the proof of Theorem \ref{ThmGTElab}.
\begin{cor}
\label{add:cor1}
If $q_n\in [0,1)$ for $n\geq 0$, $\lim_{n\to\infty}q_n =q\in [0,1),$ and $\nu_n \in \bN_{q_n,1}\to \nu_0\in \bN_{q,1}$ vaguely on $[0,1]$, then 
\begin{align}\label{add:cor1:eq1}
\lim_{n\rightarrow\infty}\sup_{x\in \mathbb{R}}\bigl|\Theta_{\nu_n}(q_n,x)-\Theta_{\nu_0}(q,x)\bigr|=0
\end{align}
and
\begin{align}\label{add:cor1:eq2}
\lim_{n\rightarrow\infty}\sup_{a\in [-1,1]}\bigl|\Lambda^\infty_{\nu_n}(q_n,a)-\Lambda^\infty_{\nu_0}(q,a)\bigr|=0.
\end{align}
\end{cor}

The proof of Proposition \ref{add:prop1} utilizes the stochastic optimal control representation for $\Theta_\gamma(t,x)$ from \cite[Corollary 2]{AC15}, which we now recall. For any $q\leq a\leq b\leq 1,$ let $\mathcal{D}_{a,b}$ be the collection of all progressively measurable processes $u=(u(s))_{a\leq s\leq b}$ with respect to the filtration generated by the standard Brownian motion $W=(W_s)_{a\leq s\leq b}$ and with $\sup_{s\in [a,b]}|u(s)|\leq 1$.  Then we can express
\begin{align}\label{rep}
\Theta_\gamma(a,x)&=\sup_{u\in \mathcal{D}_{a,b}}\e\Bigl[\Theta_\gamma\Bigl(b,x+\int_a^b\xi''\gamma u\,ds +\int_a^b\xi''^{1/2}dW_s\Bigr)-\frac{1}{2}\int_a^b\xi''\gamma u^2\,ds \Bigr].
\end{align}  
In particular, for any $\gamma\in \nN$,
\begin{align}
\Theta_{\gamma}(q,x)&=\sup_{u\in \mathcal{D}_{q,1}}\Bigl[\e\Bigl|x+\int_q^1\xi''\gamma u\,ds+\int_q^1\!\xi''^{1/2}dW_s\Bigr|-\frac{1}{2}\int_q^1\xi''\gamma \e u^2\,ds\Bigr].
\label{eqReprPsiInf}
\end{align}

\begin{remark}\label{lalRemarkAL1}\rm
Notice that the representation (\ref{eqReprPsiInf}) shows that $\lim_{x\to\pm\infty}\Theta_{\gamma}(q,x)/|x|=1$, which means that, for $a\in(-1,1),$ the minimizer in the definition of $\Lambda_\gamma^\infty(q,a)$ in (\ref{eqLambdaDef}) is unique and finite.
\end{remark}

For any $q'\in [q,1]$, denote
\begin{align*}
I_\gamma(q'):=\int_{q'}^1\!\xi''(s)\gamma(s)\,ds.
\end{align*}
We will need the following estimate on $\Theta_\gamma(q',x)$.

\begin{lem}\label{lemThEst}
	For any $\gamma\in \nN$, $q'\in [q,1]$, and $x\in \mathbb{R}$, we have
	\begin{equation}
	|x|+\frac{I_\gamma(q')}{2}\leq
	\Theta_\gamma(q',x)\leq 
	|x|+\frac{I_\gamma(q')}{2}+ \bigl(\xi'(1)-\xi'(q')\bigr)^{1/2}.
	\label{eqLemTULB}
	\end{equation}
\end{lem}
\begin{proof}
	Recall \eqref{rep} for $\Theta_\gamma(q',x)$ with $(a,b)=(q',1).$ Let $u=\mbox{sgn}(x).$ Then, by (\ref{rep}) and Jensen's inequality,
	\begin{align*}
	\Theta_\gamma(q',x)&\geq \e\Bigl[\Bigl|x+\int_{q'}^1\xi''\gamma u\,ds +\int_{q'}^1\xi''^{1/2}dW_s\Bigr|-\frac{1}{2}\int_{q'}^1\xi''\gamma u^2 \,ds \Bigr]\\
	&\geq \Bigl|x+\int_{q'}^1\xi''\gamma u\,ds \Bigr|-\frac{1}{2}\int_{q'}^1\xi''\gamma u^2 \,ds \\
	&=|x|+\int_{q'}^1\xi''\gamma \,ds -\frac{1}{2}\int_{q'}^1\xi''\gamma  \,ds 
	=|x|+\frac{I_\gamma(q')}{2}.
	\end{align*}
To establish the upper bound, for any $u\in \mathcal{D}_{q',1}$, write
	\begin{align*}
	&\Bigl|x+\int_{q'}^1\xi''\gamma u\,ds +\int_{q'}^1\xi''^{1/2}dW_s\Bigr|-\frac{1}{2}\int_{q'}^1\xi''\gamma u^2 \,ds \\
	&\leq |x|+\Bigl|\int_{q'}^1\xi''\gamma u\,ds \Bigr|+\Bigl|\int_{q'}^1\xi''^{1/2}dW_s\Bigr|-\frac{1}{2}\int_{q'}^1\xi''\gamma u^2 \,ds 
	\end{align*}
and,	using $2|u|\leq u^2+1,$ bound the second term by
	\begin{align*}
\Bigl|\int_{q'}^1\xi''\gamma u\,ds \Bigr|&\leq \int_{q'}^1\xi''\gamma |u| \,ds \leq \frac{1}{2}\int_{q'}^1\xi''\gamma u^2 \,ds +\frac{1}{2}\int_{q'}^1\xi''\gamma \,ds .
	\end{align*}
By (\ref{rep}), this implies that
	\begin{align*}
\Theta_\gamma(q',x)	
&=\sup_{u\in \mathcal{D}_{q',1}}\e\Bigl[\Bigl|x+\int_{q'}^1\xi''\gamma u\,ds +\int_{q'}^1\xi''^{1/2}dW\Bigr|-\frac{1}{2}\int_{q'}^1\xi''\gamma u^2 \,ds \Bigr]\\
	&\leq |x|+\frac{I_\gamma(q')}{2}+\e\Bigl|\int_{q'}^1\xi''^{1/2}dW\Bigr|\\
	&\leq |x|+\frac{I_\gamma(q')}{2}+\bigl(\xi'(1)-\xi'(q')\bigr)^{1/2}.
	\end{align*}
Taking the supremum over $u$ gives the desired upper bound.
\end{proof}

\subsection{Proof of Proposition \ref{add:prop1}} By the definition of $\Lambda_{\nu}^\infty$, the assertion \eqref{add:prop1:eq2} evidently follows from \eqref{add:prop1:eq1}, so we only focus on proving \eqref{add:prop1:eq1}. 	Obviously this assertion holds  if $q=1.$ From now on, assume that $q\in [0,1).$ 

Let $\gamma_n,\Delta_n$ and $\gamma_0,\Delta_0$ be the pairs associated with $\nu_{n}$ and $\nu_0$ respectively. From the vague convergence, $\gamma_n(s)\to \gamma(s)$ almost surely on $[0,1).$ Therefore, for any $q'\in [q,1),$ 
$$
\sup_{n\geq 1}\sup_{s\in [q,q']}\gamma_n(s)<\infty,
$$ 
which yields, by the bounded convergence theorem,
\begin{align}\label{add:lem1:proof:eq-1}
\lim_{n\rightarrow\infty}\int_0^{q'}\! |\gamma_n(s)-\gamma_0(s)|\,ds=0.
\end{align}
(However, of course, it is not necessarily true that $\Delta_n\to \Delta_0.$)

Next, fix $q'\in [q,1)$. For any $u\in \mathcal{D}_{q,q'}$ and $\gamma\in \nN$, set $$
\Gamma_{\gamma}(q',u):=
\e\Bigl[\Bigl|x+\int_{q}^{q'}\!\xi''\gamma u\,ds+\int_{q}^{q'}\!dB_s\Bigr|-\frac{1}{2}\int_q^{q'}\xi''\gamma u^2\,ds \Bigr],
$$ 
where $dB_s:=\xi''(s)^{1/2}dW_s.$ Using \eqref{rep} for $\Theta_{\gamma_n}(q,x)$ with $(a,b)=(q,q')$ and Lemma \ref{lemThEst},
\begin{align}\label{add:lem1:proof:eq0}
\Bigl|\Theta_{\gamma_n}(q,x)-\frac{1}{2}I_{\gamma_n}(q')-\sup_{u\in \mathcal{D}_{q,q'}}\Gamma_{\gamma_n}(q',u)\Bigr|\leq \bigl(\xi'(1)-\xi'(q')\bigr)^{1/2}.
\end{align}
In addition, by the triangle inequality,
\begin{align}\label{add:lem1:proof:eq1}
\Bigl|\Gamma_{\gamma_n}(q',u)-\Gamma_{\gamma_0}(q',u)\Bigr|\leq \frac{3}{2}\int_q^{q'}\!\xi''|\gamma_n-\gamma_0|\,ds.
\end{align}
For any $v\in \mathcal{D}_{q,1}$, if we write $u=v1_{[q,q']}$ then
\begin{align*}
\Gamma_{\gamma_0}(1,v)
&=\e\Bigl[\Bigl|x+\int_{q}^{q'}\!\xi''\gamma_0 u\, ds+\int_{q}^{q'}\!dB_s+\int_{q'}^1\!\xi''\gamma_0 v\,ds+\int_{q'}^1\!dB_s\Bigr|\\
&\qquad-\frac{1}{2}\int_q^{q'}\!\xi''\gamma_0 u^2\,ds-\frac{1}{2}\int_{q'}^{1}\!\xi''\gamma_0 v^2 \,ds\Bigr],
\end{align*}
which, by the triangle inequality and $\e \bigl(\int_{q}^{q'}\!dB_s\bigr)^2=\xi'(1)-\xi'(q')$ implies that
$$
\Bigl|\Gamma_{\gamma_0}(1,v)-\Gamma_{\gamma_0}(q',u)\Bigr|\leq \frac{3}{2}\int_{q'}^1\!\xi''\gamma_0\, ds+\bigl(\xi'(1)-\xi'(q')\bigr)^{1/2}.
$$
From this inequality, \eqref{add:lem1:proof:eq0} and \eqref{add:lem1:proof:eq1}, we see that by taking maximum over $v\in \mathcal{D}_{q,1}$ and using \eqref{rep} for $\Theta_{\gamma}(q,x)$ with $(a,b)=(q,1)$, it follows that
\begin{align*}
&\Bigl|\Theta_{\gamma_n}(q,x)-\frac{1}{2}I_{\gamma_n}(q')-\Theta_{\gamma_0}(q,x)\Bigr|\\
&\leq  \frac{3}{2}\int_{q'}^1\xi''\gamma_0 \,ds +\frac{3}{2}\int_q^{q'}\!\xi''|\gamma_n-\gamma_0|\,ds+2\bigl(\xi'(1)-\xi'(q')\bigr)^{1/2}.
\end{align*}
Taking a limit and using \eqref{add:lem1:proof:eq-1} gives
\begin{align}
\begin{split}\label{add:lem1:proof:eq2}
&\limsup_{n\rightarrow\infty}\sup_{x\in \mathbb{R}}\Bigl|\Theta_{\gamma_n}(q,x)-\frac{1}{2}I_{\gamma_n}(q')-\Theta_{\gamma_0}(q,x)\Bigr|
\\
&\leq \frac{3}{2}\int_{q'}^1\!\xi''\gamma_0 \,ds+2\bigl(\xi'(1)-\xi'(q')\bigr)^{1/2}.
\end{split}
\end{align}
Note that, since
\begin{align*}
I_{\gamma_n}(q')&=\int_{q'}^1\!\xi''(s)\gamma(s)\,ds
=\int \xi''d\nu_n-\xi''(1)\Delta_n-\int_q^{q'}\!\xi''\gamma_n\,ds,
\end{align*}
we can rewrite the expression on the left-hand side of the above inequality as
\begin{align*}
&\Theta_{\gamma_n}(q,x)-\frac{1}{2}I_{\gamma_n}(q')-\Theta_{\gamma_0}(q,x)\\
&=\Theta_{\nu_n}(q,x)-\Theta_{\gamma_0}(q,x)-\frac{1}{2}\Bigl(\int \xi''d\nu_n-\int_q^{q'}\xi''\gamma_n\,ds \Bigr)\\
&=\Theta_{\nu_n}(q,x)-\Theta_{\nu_0}(q,x)-\frac{1}{2}\Bigl(\int \xi''d\nu_n-\int_q^{q'}\xi''\gamma_n\,ds -\xi''(1)\Delta_0\Bigr).
\end{align*}
From the vague convergence $\nu_n\to\nu_0$ and \eqref{add:lem1:proof:eq-1}, the last term converges to
$$
\frac{1}{2}\Bigl(
\int \xi''d\nu_0-\int_q^{q'}\xi''\gamma_0\,ds -\xi''(1)\Delta_0
\Bigr)
=
\frac{1}{2}\int_{q'}^1\xi''\gamma_0\,ds 
$$
and, therefore, \eqref{add:lem1:proof:eq2} implies
\begin{align*}
&\limsup_{n\rightarrow\infty}\sup_{x\in \mathbb{R}}\Bigl|\Theta_{\nu_n}(q,x)-\Theta_{\nu_0}(q,x)\Bigr| \leq 2\int_{q'}^1\xi''\gamma_0 \,ds +2\bigl(\xi'(1)-\xi'(q')\bigr)^{1/2}.
\end{align*}
The right-hand side vanishes as $q'\uparrow 1$, which completes the proof.
\qed

\subsection{Proof of Corollary \ref{add:cor1}}
	Let $\gamma_n':=1_{[q,1]}\gamma_n$ and $\nu_n'(A):=\nu_n([q,1]\cap A).$ Then $\nu_n'$ converges to $\nu_0$ vaguely and from Proposition \ref{add:prop1},
	\begin{align}
	\begin{split}\label{add:cor1:proof:eq1}
\lim_{n\rightarrow\infty}\sup_{x\in \mathbb{R}}\bigl|\Theta_{\nu_n'}(q,x)-\Theta_{\nu_0}(q,x)\bigr|&=0,\\
\lim_{n\rightarrow\infty}\sup_{a\in [-1,1]}\bigl|\Lambda^\infty_{\nu_n'}(q,a)-\Lambda^\infty_{\nu_0}(q,a)\bigr|&=0.
\end{split}
\end{align}
On the other hand, by using the representation \eqref{rep} for $\Theta_{\gamma_n'}(q,x)$ with $(a,b)=(q,1)$ and $\Theta_{\gamma_n}(q_n,x)$ with $(a,b)=(q_n,1)$, we see that
\begin{align*}
\Theta_{\gamma_n'}(q,x)&=\sup_{u\in \mathcal{D}_{q,1}}\e\Bigl[\Bigl|x+\int_q^1\xi''\gamma_n u\,ds +\int_q^1\xi''^{1/2}dW_s\Bigr|-\frac{1}{2}\int_q^1\xi''\gamma_n u^2\,ds \Bigr],\\
\Theta_{\gamma_n}(q_n,x)&=\sup_{u\in \mathcal{D}_{q_n,1}}\e\Bigl[\Bigl|x+\int_{q_n}^1\xi''\gamma_n u\,ds +\int_{q_n}^1\xi''^{1/2}dW_s\Bigr|-\frac{1}{2}\int_{q_n}^1\xi''\gamma_n u^2\,ds \Bigr].
\end{align*}
From these, we see that
\begin{align*}
&\sup_{x\in \mathbb{R}}\bigl|\Theta_{\gamma_n'}(q,x)-\Theta_{\gamma_n}(q_n,x)\bigr|
\leq \frac{3}{2}\int_{q_n\wedge q}^{q_n\vee q}\xi''\gamma_n\,ds +\e\Bigl|\int_{q_n\wedge q}^{q_n\vee q}\xi''^{1/2}dW_s\Bigr|
\\
&\leq \frac{3\xi''(1)}{2}|q_n-q|\max_{q_n\wedge q\leq s\leq q_n\vee q}\gamma_n(s)+\e\Bigl|\int_{q_n\wedge q}^{q_n\vee q}\xi''^{1/2}dW_s\Bigr|.
\end{align*}
Note that from the vague convergence of $\nu_n$ to $\nu_0$, $\gamma_n$ converges to $\gamma_0$ a.s. Using the fact that $\gamma_n$ are nondecreasing, we see that
$$
\sup_{n\geq 1}\max_{q_n\wedge q\leq s\leq q_n\vee q}\gamma_n(s)<\infty.
$$
Consequently,
\begin{align*}
\lim_{n\to\infty}\sup_{x\in \mathbb{R}}\bigl|\Theta_{\gamma_n'}(q,x)-\Theta_{\gamma_n}(q_n,x)\bigr|&=0,\\
\lim_{n\to\infty}\sup_{a\in [-1,1]}\bigl|\Lambda^\infty_{\gamma_n'}(q,a)-\Lambda^\infty_{\gamma_n}(q_n,a)\bigr|&=0.
\end{align*}
This together with \eqref{add:cor1:proof:eq1} completes our proof.
\qed

\section{Uniqueness of the minimizer}

This section is devoted to the proof of Theorem \ref{zero:thm4}. We begin with the following two lemmas which will be needed in the proof. For any fixed measure $\mu\in M_*$ with $q=\int a^2d\mu(a)$, it was proved in \cite{GenTAP} that the functional $\zeta\to\nTAP^\beta(\mu,\zeta)$ has a minimizer $\zeta_{\beta,\mu}$ in $\mathcal{M}_{0,1}$ and the restriction of this minimizer to $[q,1]$ (which can be viewed as an element of  $\mathcal{M}_{q,1}$) is unique. 

Recall the stochastic optimal control representation for $\Phi_\zeta^\beta,$ which states that for any $\zeta\in \mathcal{M}_{q,1}$, one can express $$
\Phi_\zeta^\beta(q,x)=\sup_{u}\Bigl[\e\log 2\cosh\Bigl(x+\int_q^1\!\beta^2\xi''\zeta u\,ds+\int_q^1\!\beta\xi''^{1/2}\,dW_s\Bigr)-\frac{\beta^2}{2}\int_q^1\!\xi''\zeta \e u^2\,ds\Bigr],
$$ 
where the supremum is taken over all progressively measurable processes $u$ on $[q,1]$ with respect to the standard Brownian motion $W.$ In particular, the supremum here is attained by 
\begin{align}
u_{x,\zeta}^\beta(s)=\partial_x\Phi_{\zeta}^\beta(q,X_{x,\zeta}^\beta(s)),
\label{eqUoptBeta}
\end{align}
where $X_{x,\zeta}^\beta$ is the strong solution of
\begin{align}
\label{eq:dX}
dX_{x,\zeta}^\beta(s)&=\beta^2\xi''(s)\zeta(s)\partial_x\Phi_{\zeta}^\beta(s,X_{x,\zeta}^\beta(s))\,ds +\beta\xi''(s)^{1/2}dW_s
\end{align}
with the initial condition $X_{x,\zeta}^\beta(q)=x.$ 

\begin{lem}
	For any $\zeta\in \mathcal{M}_{q,1}$ and $x\in \mathbb{R}$, we have that
	\begin{align*}
	\partial_\beta\Phi_{\zeta}^\beta(q,x)=\beta\Bigl(\xi'(1)-\xi'(q)-\int_{q}^1(\xi'(s)-\xi'(q))\e u_{x,\zeta}^\beta(s)^2d\zeta(s)\Bigr).
	\end{align*}
\end{lem}

\begin{proof}
	Let $\alpha$ be any nondecreasing function on $[a,b]$ with right-continuity for some $0\leq a<b\leq 1.$ For any $f,g$ continuously differentiable functions on $[a,b]$, the following integration by parts is valid,
	\begin{align}
	\begin{split}\label{int}
	\int_a^bg'(s)f(s)\alpha(s)ds
	&=g(b)f(b)\alpha(b)-g(a)f(a)\alpha(a)\\
	&\quad-\int_{a}^bg(s)f(s)d\alpha(s)-\int_a^bg(s)f'(s)\alpha(s)ds,
	\end{split}
	\end{align}
	where the first integral on the right-hand side should be understood as the Riemann-Stieltjes integral.	Note that a direct differentiation of the Parisi PDE in $\beta$ gives
	\begin{align*}
	\partial_s\partial_\beta\Phi_{\zeta}^\beta&=-\frac{\beta^2\xi''}{2}\bigl(\partial_{xx}\partial_\beta\Phi_{\zeta}^\beta+2\zeta\bigl(\partial_x\Phi_\zeta^\beta\bigr)\bigl(\partial_x\partial_\beta\Phi_\zeta^\beta\bigr)\bigr)-\beta\xi''\bigl(\partial_{xx}\Phi_{\zeta}^\beta+\zeta\bigl(\partial_x\Phi_\zeta^\beta\bigr)^2\bigr).
	\end{align*}
	From the Feynman-Kac formula,
	\begin{align*}
	\partial_\beta\Phi_{\zeta}^\beta(q,x)&=\int_q^1\beta\xi''(s)\e \Bigl[\partial_{xx}\Phi_{\zeta}^\beta(s,X_{x,\zeta}^\beta(s))+\zeta\bigl(\partial_x\Phi_\zeta^\beta(s,X_{x,\zeta}^\beta(s))\bigr)^2\Bigr]ds.
	\end{align*}
	For convenience, from now on, we denote $u(s)=\partial_x\Phi_{\zeta}^\beta(s,X_{x,\zeta}^\beta(s))$ and $v(s)=\partial_{xx}\Phi_{\zeta}^\beta(s,X_{x,\zeta}^\beta(s)).$ Using the usual integration by part gives
	\begin{align*}
	\int_q^1\xi''(s)\e v(s)ds&=\xi'(1)\e v(1)-\xi'(q)\e v(q)+\beta^2\int_q^1\xi'(s)\xi''(s)\zeta(s)\e v(s)^2ds\\
	&=\xi'(1)(1-\e u(1)^2)-\xi'(q)\e v(q)+\beta^2\int_q^1\xi'(s)\xi''(s)\zeta(s)\e v(s)^2ds,
	\end{align*}
	where the second equality used the fact that $v(1)=1-u(1)^2$. In addition, from \eqref{int},
	\begin{align*}
	\int_q^1\xi''(s)\zeta(s) \e u(s)^2ds&=\xi'(1)\e u(1)^2-\xi'(q)\e u(q)^2\\
	&\quad-\int_{q}^1\xi'(s)\e u(s)^2d\zeta(s)-\beta^2\int_q^1\xi'(s)\xi''(s)\zeta(s)\e v(s)^2ds.
	\end{align*}
	These imply that
	\begin{align*}
	\partial_\beta\Phi_{\zeta}^\beta(q,x)&=-\beta\xi'(q)\bigl(\e v(q)+\zeta(q)\e u(q)^2\bigr)+\beta\Bigl(\xi'(1)-\int_{q}^1\xi'(s)\e u(s)^2d\zeta(s)\Bigr).
	\end{align*}
	Finally, our proof is completed by plugging the following equation (see \cite[Lemma 37]{GenTAP}) into this equation, 
	$$
	\e v(q)+\zeta(q)\e u(q)^2=1-\int_{q}^1\e u(s)^2d\zeta(s).
	$$
\end{proof}

\begin{lem}\label{zero:lem9}
For any $\beta>0$ and $\mu\in M_*,$ we have that
\begin{align}\label{zero:lem9:eq1}
\beta\int_q^1\bigl(\xi'(s)-\xi'(q)\bigr)\zeta_{\beta,\mu}(s)ds&\leq\nTAP^\infty(\mu).
\end{align}
Furthermore, if $\mu$ is supported on $[-1+\eta,1-\eta]$ for some $\eta\in (0,1)$, then 
\begin{align}\label{zero:lem9:eq2}
\frac{d}{d\beta}\nTAP^\beta(\mu)=\beta\int_q^1\bigl(\xi'(s)-\xi'(q)\bigr)\zeta_{\beta,\mu}(s)\,ds\rightarrow \nTAP^\infty(\mu),\,\,\mbox{as $\beta\to\infty.$}
\end{align}
\end{lem}

\begin{proof}[Proof of Lemma \ref{zero:lem9}] If $\mu=\delta_1$, the inequality \eqref{zero:lem9:eq1}, obviously, holds. From now on, we assume that $\mu\neq \delta_1,$ so $q=\int a^2d\mu(a)<1.$ First, let us explain that it is enough to prove the assertion \eqref{zero:lem9:eq1} for measures $\mu$ with the support in $(-1,1)$. On the one hand, we noted in the proof of Theorem \ref{ThmTAPrepZT} that $\nTAP^\infty(\mu)$ is continuous in $\mu$ and, moreover, we can approximate any $\mu$ by measures with the support in $(-1,1)$ while keeping $q=\int a^2d\mu(a)$ fixed. On the other hand, it was shown in the proof of Theorem 10 $(ii)$ in \cite{GenTAP} that $\nTAP^{\beta}(\mu,\zeta)$ is continuous in $\mu$ for any fixed $\zeta\in\mathcal{M}_{0,1}$ and, by the properties of the Parisi functional $\Phi_\zeta$, it is $L^1$-Lipschitz in $\zeta$ uniformly over $\mu$, which  implies that $(\mu,\zeta)\to\nTAP^{\beta}(\mu,\zeta)$ is continuous. By the uniqueness of the minimizer restricted to $[q,1]$, this implies that $\zeta_{\beta,\mu}$ is also continuous in $\mu$ restricted by $q=\int a^2d\mu(a).$ These observations imply that it is enough to prove Lemma \ref{zero:lem9} for $\mu$ with the support in $(-1,1)$. From now on, we suppose that $\supp(\mu)\subseteq [-(1-\eta),1-\eta]$ for some $\eta>0$.

Fix $\beta>0$. For any $h\geq 0,$
\begin{align}
\begin{split}\label{eqTAPBHlower}
&\nTAP^{\beta}(\mu)-\nTAP^{\beta-h}(\mu)
\geq \nTAP^{\beta}(\mu,\zeta_{\beta,\mu})-\nTAP^{\beta-h}(\mu,\zeta_{\beta,\mu})
\\
&=\int\Bigl(\Lambda_{\zeta_{\beta,\mu}}^{\beta}(q,a)-\Lambda_{\zeta_{\beta,\mu}}^{\beta-h}(q,a)\Bigr)d\mu(a)-\frac{\beta^2-(\beta-h)^2}{2}\int_q^1s\xi''\zeta_{\beta,\mu}\, ds
\end{split}
\end{align}	
and
\begin{align}
\begin{split}\label{eqTAPBHlower2}
&\nTAP^{\beta+h}(\mu)-\nTAP^{\beta}(\mu)
\leq\nTAP^{\beta+h}(\mu,\zeta_{\beta,\mu})-\nTAP^{\beta}(\mu,\zeta_{\beta,\mu})
\\
&=\int\Bigl(\Lambda_{\zeta_{\beta,\mu}}^{\beta+h}(q,a)-\Lambda_{\zeta_{\beta,\mu}}^{\beta}(q,a)\Bigr)d\mu(a)-\frac{(\beta+h)^2-\beta^2}{2}\int_q^1s\xi''\zeta_{\beta,\mu}\, ds.
\end{split}
\end{align}	
Note that, for any $a\in (-1,1)$,
\begin{align}
\begin{split}
\Lambda_{\zeta_{\beta,\mu}}^{\beta}(q,a)-\Lambda_{\zeta_{\beta,\mu}}^{\beta-h}(q,a)&\geq \Phi_{\zeta_{\beta,\mu}}^{\beta}(q,x(a))-\Phi_{\zeta_{\beta,\mu}}^{\beta-h}(q,x(a)),\\
\Lambda_{\zeta_{\beta,\mu}}^{\beta+h}(q,a)-\Lambda_{\zeta_{\beta,\mu}}^{\beta}(q,a)&\leq \Phi_{\zeta_{\beta,\mu}}^{\beta+h}(q,x(a))-\Phi_{\zeta_{\beta,\mu}}^{\beta}(q,x(a)),
\end{split}\label{eqLamPhiComp}
\end{align}
where $x(a)$ is the minimizer of 
$$
\Lambda_{\zeta_{\beta,\mu}}^{\beta}(q,a)=\inf_x\Bigl(\Phi_{\zeta_{\beta,\mu}}^{\beta}(q,x)-xa\Bigr).
$$ 
It was proved in Section 12.2 in \cite{GenTAP} that $x(a)$ is continuous and bounded on $[0,1-\eta]$. Recall from Proposition 4 in \cite{ACdual} (with $\gamma=\beta^2$ there) that
\begin{align*}
\frac{d}{d\beta}\Phi_{\zeta}^{\beta}(q,x)&=\beta\Bigl(\xi'(1)-\xi'(q)-\int_{[q,1]}(\xi'(s)-\xi'(q))\e u_{x,\zeta}(s)^2d\zeta(s)\Bigr),
\end{align*}
where $u_{x,\zeta}^{\beta}(s)$ was defined in (\ref{eqUoptBeta}). If we denote
$$
f^\beta(a,s):=\e u_{x(a),\zeta_{\beta,\mu}}^{\beta}(s)^2,
$$
then
\begin{align*}
\int\frac{d}{d\beta}\Phi_{\zeta_{\beta,\mu}}^{\beta}(q,x(a))d\mu(a)&=\beta\Bigl(\xi'(1)-\xi'(q)-\int_{q}^1(\xi'(s)-\xi'(q))\int f^\beta(a,s)d\mu(a)d\zeta_{\beta,\mu}(s)\Bigr).
\end{align*}
To handle this equation,  for any $\zeta\in \mathcal{M}_{q,1}$ and $\theta\in [0,1]$, set $\zeta_\theta=(1-\theta)\zeta_{\beta,\mu}+\theta\zeta.$ By a standard calculation (see e.g. \cite{C17}), one can compute the directional derivative of $\nTAP^\beta$,
\begin{align*}
\frac{d}{d\theta}\nTAP^\beta(\mu,\zeta_\theta)\Big|_{\theta=0^+}&=\frac{\beta^2}{2}\int_q^1\xi''(s)(\zeta(s)-\zeta_{\beta,\mu}(s))\Bigl(\int\! f^\beta(a,s)\,d\mu(a)-s\Bigr)\,ds ,
\end{align*}
which must be non-negative by the minimality of $\zeta_{\beta,\mu}.$ Again, in a standard way one can readily see (by varying $\zeta$) that this forces $\int\! f^\beta(a,s)\, d\mu(a)=s$ for any $s\geq q$ in the support of $\zeta_{\beta,\mu}$. This implies that
\begin{align*}
\int\frac{d}{d\beta}\Phi_{\zeta_{\beta,\mu}}^{\beta}(q,x(a))d\mu(a)&=\beta\Bigl(\xi'(1)-\xi'(q)-\int_q^1(\xi'(s)-\xi'(q))sd\zeta_{\beta,\mu}(s)\Bigr).
\end{align*}
Here, note that from \eqref{int},
\begin{align*}
\int_{q}^1\xi''(s)s\zeta_{\beta,\mu}(s)ds
&=\xi'(1)-\xi'(q)-\int_q^1(\xi'(s)-\xi'(q))sd\zeta_{\beta,\mu}(s)-\int_q^1(\xi'(s)-\xi'(q))\zeta_{\beta,\mu}(s)ds.
\end{align*}
Plugging these two equations  into the previous display leads to
\begin{align*}
\int\frac{d}{d\beta}\Phi_{\zeta_{\beta,\mu}}^{\beta}(q,x(a))d\mu(a)&=\beta\int_{q}^1\xi''(s)s\zeta_{\beta,\mu}(s)ds+\beta\int_q^1\bigl(\xi'(s)-\xi'(q)\bigr)\zeta_{\beta,\mu}(s)ds.
\end{align*}
From this,  (\ref{eqTAPBHlower}), \eqref{eqTAPBHlower2}, and (\ref{eqLamPhiComp}) (together with our assumption that $\supp(\mu)\subset (-1,1)$) it follows that the left and right derivatives of $\nTAP^\beta(\mu)$ (which exist from convexity in $\beta$) satisfy
\begin{align*}
D_\beta^-\nTAP^\beta(\mu)&\geq
\beta\int_q^1\bigl(\xi'(s)-\xi'(q)\bigr)\zeta_{\beta,\mu}(s)ds,\\
D_\beta^+\nTAP^\beta(\mu)&\leq
\beta\int_q^1\bigl(\xi'(s)-\xi'(q)\bigr)\zeta_{\beta,\mu}(s)ds.
\end{align*}
Now, since $\nTAP^\beta(\mu)$ is a convex function in $\beta$, this implies that
$$
\frac{d}{d\beta}\nTAP^\beta(\mu)=\beta\int_q^1\bigl(\xi'(s)-\xi'(q)\bigr)\zeta_{\beta,\mu}(s)ds.
$$
From this and the convexity of $\nTAP^\beta(\mu)$ in $\beta$, the assertion \eqref{zero:lem9:eq1} follows by noting that
\begin{align*}
\frac{d}{d\beta}\nTAP^\beta(\mu)\leq \lim_{\beta\to \infty}\frac{\nTAP^{\beta}(\mu)}{\beta}=\nTAP^\infty(\mu),
\end{align*}
while the assertion \eqref{zero:lem9:eq2} is validated by using the above inequality and  
$$\nTAP^\infty(\mu)=\lim_{\beta\to\infty}\frac{\nTAP^\beta(\mu)-\nTAP^0(\mu)}{\beta}\leq \lim_{\beta\rightarrow\infty}\frac{d}{d\beta}\nTAP^\beta(\mu).$$
This finishes the proof.
\end{proof}

\subsection*{Proof of Theorem \ref{zero:thm4}} Let $\mu$ be fixed and set $q=\int a^2d\mu(a)$. In the case that $q=1$, the space $\nN$ is a singleton and the theorem follows trivially. From now on, assume that $q<1$. Denote by $\zeta_{\beta,\mu}$ the minimizer associated to $\nTAP^\beta(\mu).$ Note that, by Lemma \ref{zero:lem9} above, 
	\begin{align}\label{zero:thm4:eq-1}
\beta\int_q^1\! (\xi'(s)-\xi'(q))\zeta_{\beta,\mu}(s)\,ds\leq \nTAP^\infty(\mu),\,\,\forall \beta>0.
\end{align}
Denote $\gamma_{\beta,\mu}:=\beta\zeta_{\beta,\mu}$ and, for all measurable sets $A\subset [q,1],$ set 
$$
\nu_{\beta,\mu}(A)=\int_A\! \gamma_{\beta,\mu}(s)\,ds.
$$
Since $\zeta_{\beta,\mu}$ is nondecreasing, \eqref{zero:thm4:eq-1} implies that 
$$
\gamma_{\beta,\mu}(s)\leq \frac{\nTAP^\infty(\mu)}{\xi(1)-\xi(s)-\xi'(q)(1-s)},\,\,\forall s\in [q,1).
$$
On the other hand, from this inequality and \eqref{zero:thm4:eq-1}, we also see that $\sup_{\beta>0}\int_q^1\gamma_{\beta,\mu} \,ds <\infty.$
Because of these, we can choose a subsequence of $\beta\uparrow\infty$ so that $\gamma_{\beta,\mu}$ converges to some $\gamma_{\mu}$ vaguely on $[q,1)$ and $\int_q^1\gamma_{\beta,\mu} \,ds$ is convergent. For notational clarity, we will assume that $\gamma_{\beta,\mu}$ converges to $\gamma_{\mu}$ vaguely on $[q,1)$ and $\int_q^1\gamma_{\beta,\mu} \,ds$ converges without going to a subsequence. Note that since $\gamma_{\beta,\mu}(s)\to\gamma_{\mu}(s)$ almost surely on $[q,1]$, by Fatou's lemma, $\int_q^1\gamma_{\mu}\, ds<\infty$, which means that $\gamma_{\mu}\in \nN.$ Furthermore, if we denote 
$$
\Delta:=\lim_{\beta\rightarrow\infty}\nu_{\beta,\mu}([q,1])-\int_q^1\!\gamma_{\mu}\, ds,
$$
and define $\nu$ by 
$$
\nu(A):=\int_A\!\gamma_{\mu}\,ds+\Delta \delta_1(A)
$$
then $\nu_{\beta,\mu}$ converges to $\nu$ vaguely on $[q,1].$ Indeed, for any continuous function $\phi$ on $[q,1]$ with $\sup_{s\in [q,1]}|\phi(s)|\leq 1,$ 
\begin{align*}
\Bigl|\int_{q}^1 \phi d(\nu_{\beta,\mu}-\nu)\Bigr|
&\leq \int_{q}^{q'} |\gamma_{\beta,\mu}-\gamma|\,ds +\Bigl|\int_{q'}^1 (\phi-1) d(\nu_{\beta,\mu}-\nu)\Bigr|+\Bigl|\int_{q'}^1 d(\nu_{\beta,\mu}-\nu)\Bigr|\\
&\leq \int_{q}^{q'} |\gamma_{\beta,\mu}-\gamma|\,ds +\sup_{s\in [q',1]}|\phi(s)-1|\int_{q'}^1(\gamma_{\beta,\mu}+\gamma) \,ds +\Bigl|\int_{q'}^1 d(\nu_{\beta,\mu}-\nu)\Bigr|
\end{align*}
and, passing to the limit,
\begin{align*}
\limsup_{q'\uparrow 1}\limsup_{\beta\rightarrow\infty}\Bigl|\int_{q}^1 \phi d(\nu_{\beta,\mu}-\nu)\Bigr|&\leq \limsup_{q'\uparrow 1}\limsup_{\beta\rightarrow\infty}\Bigl|\int_{q'}^1 d(\nu_{\beta,\mu}-\nu)\Bigr|=0,
\end{align*}
where the right-hand side vanishes because, for any $q'\in [q,1),$
\begin{align*} 
\int_{q'}^1 d(\nu_{\beta,\mu}-\nu)&=\int_{q}^1\gamma_{\beta,\mu} \,ds -\int_q^1\gamma_{\mu}\,ds -\Delta+\int_q^{q'}(\gamma_{\beta,\mu} -\gamma_{\mu})\,ds \rightarrow 0\,\,\mbox{as $\beta\to \infty.$}
\end{align*}

Next we prove that $\gamma_{\mu}$ is a minimizer to $\nTAP^\infty(\mu).$ From Proposition \ref{add:prop1}, $$
\lim_{\beta\to\infty}\sup_{a\in [-1,1]}\Bigl|\Lambda_{\gamma_{\beta,\mu}}^\infty(q,a)-\Lambda_{\gamma_{\mu}}^\infty(q,a)-\frac{\Delta\xi''(1)}{2}\Bigr|=\lim_{\beta\to\infty}\sup_{a\in [-1,1]}\Bigl|\Lambda_{\nu_{\beta,\mu}}^\infty(q,a)-\Lambda_{\nu}^\infty(q,a)\Bigr|=0.
$$
Also, note that from the vague convergence of $\nu_{\beta,\mu}$ to $\nu,$
$$
\int_q^1\xi''s \gamma_{\beta,\mu} \,ds=\int_q^1\xi'' s\,d\nu_{\beta,\mu}\to \int_q^1\xi'' s\,d\nu=\int_q^1\xi''s \gamma_{\mu} \,ds+\frac{\Delta\xi''(1)}{2}.
$$
Together these lead to
\begin{align}\label{zt:add:eq7}
\lim_{\beta\to \infty}\nTAP^\infty(\mu,\gamma_{\beta,\mu})&=\nTAP^\infty(\mu,\gamma_{\mu}).
\end{align}
Since, from \eqref{zero:eq1},
\begin{align}\label{zt:add:eq8}
\Bigl|\nTAP^\infty(\mu,\gamma_{\beta,\mu})-\frac{1}{\beta}\nTAP^\beta(\mu,\zeta_{\beta,\mu})\Bigr|\leq \frac{\log 2}{\beta}
\end{align}
and, from Lemma \ref{zero:lem1},
\begin{align}\label{zt:add:eq9}
\frac{1}{\beta}\nTAP^\beta(\mu)=\frac{1}{\beta}\nTAP^\beta(\mu,\zeta_{\beta,\mu})\rightarrow \nTAP^\infty(\mu),
\end{align}
we conclude that
$
\nTAP^\infty(\mu)\geq \nTAP^\infty(\mu,\gamma_{\mu}).
$
Hence, $\gamma_{\mu}$ is a minimizer to $\nTAP^\infty(\mu).$

Finally, we show that the minimizer to $\nTAP^\infty(\mu)$ is unique. To see this, we recall from Lemma 5 in \cite{CHL} that $\Theta_{\gamma}(q,x)$ is a strictly convex functional in $(\gamma,x)\in \nN\times\mathbb{R}$. This implies that for any $a\in (-1,1)$, $\Lambda^\infty_\gamma(q,a)$ is strictly convex in $\gamma$ and so is $\nTAP^\infty(\mu,\gamma).$ Hence, $\nTAP^\infty(\mu)$ has a unique minimizer, $\gamma_{\mu}.$
\qed

\begin{remark}\rm \label{rm1}
Recall the measures $\nu_{\beta,\mu}$ and $\nu$ in the above proof. From  (\ref{zero:lem9:eq2}), we see that
\begin{align*}
\nTAP^\infty(\mu)&=\lim_{\beta\to\infty}\int_q^1(\xi'(s)-\xi'(q))d\nu_{\beta,\mu}(s)=\int_q^1(\xi'(s)-\xi'(q))d\nu(s)\\
&=\int_q^1(\xi'(s)-\xi'(q))\gamma_0(s)\,ds +(\xi'(1)-\xi'(q))\Delta.
\end{align*}	
Moreover, we showed that $\gamma_{\beta,\mu}(s)=\beta\zeta_{\beta,\mu}(s)$ converges to $\gamma_\mu(s)$ almost surely on $[q,1)$ as $\beta\to\infty.$
\end{remark}

\section{Energy of TAP states}

In this section, we will prove Theorem \ref{zt:thm2}.

\begin{proof}[Proof of Theorem \ref{zt:thm2}]
	Let us denote 
	\begin{equation}
	f_N(\beta):=\sup_{m\in S_N(q)}f_m(\beta)
	=\sup_{m\in S_N(q)}\Bigl(\frac{\beta H_N(m)}{N}+\nTAP^\beta(\mu_m)\Bigr).
	\end{equation}
	Recall from Theorem \ref{thm:GenTAP} that for any $q$ in the support of the Parisi measure $\zeta_{\beta}^*$, the following limits exist almost surely (using Borell's inequality and the concentration of the free energy),
	\begin{align}\label{zt:ex:eq4}
	\PP(\beta):=\lim_{N\to\infty} f_N(\beta)=\lim_{N\to\infty}\bigl(\beta F_N(\beta)\bigr)
	=\beta \PP_\beta(\zeta_{\beta}^*)
	\end{align}
	and, by \cite[Remark 1]{ACdual}, $\PP(\beta)$ is differentiable with
	\begin{align}\label{zt:ex:eq5}
	\PP'(\beta)=\frac{d}{d\beta}\lim_{N\to\infty}\bigl(\beta F_N(\beta)\bigr)
	&=\beta\int_0^1\!\xi'(s)\zeta_{\beta}^*(s)\,ds.
	\end{align} 
	Since $\nTAP^\beta_{N,n}(m,\eps,\delta)$ is convex in $\beta$ and, by Theorem \ref{thm:TAPcorrection}, it converges to $\nTAP^\beta(\mu_m)$ uniformly in $m\in [-1,1]^N$, it follows that, for any $\mu\in M_*$, $\nTAP^\beta(\mu)$ is convex in $\beta>0$, which implies that $f_m(\beta)$ and $f_N(\beta)$ are convex in $\beta$. Since
	$$
	f_N(\beta\pm h)=\max_{m\in S_N(q)}f_m(\beta\pm h)\geq 
	\max_{m\in M_{\beta,q}(\eps_N)}f_m(\beta\pm h),
	$$
	for any $m\in M_{\beta,q}(\eps_N)$ and $h>0$, we can write
	\begin{align*}
	&\frac{f_N(\beta+h)-f_N(\beta)}{h}\geq \frac{f_m(\beta+h)-f_m(\beta)-\eps_N}{h}
	\geq f_m'(\beta)-\frac{\eps_N}{h},
	\\
	&\frac{f_N(\beta)-f_N(\beta-h)}{h}\leq \frac{f_m(\beta)-f_m(\beta-h)+\eps_N}{h}
	\leq f_m'(\beta)+\frac{\eps_N}{h},
	\end{align*}
	using convexity in the last inequality in each line, where the existence of $f_m'(\beta)$ is guaranteed by \eqref{zero:lem9:eq2}. Taking the supremum in the first line and infimum in the second line over $m\in M_{\beta,q}(\eps_N)$ and taking limits,
	\begin{align*}
	\frac{\PP(\beta)-\PP(\beta-h)}{h}
	&\leq
	\liminf_{N\to\infty}\inf_{M_{\beta,q}(\eps_N)} f_m'(\beta)
	\\
	&\leq
	\limsup_{N\to\infty} \sup_{M_{\beta,q}(\eps_N)} f_m'(\beta)
	\leq \frac{\PP(\beta+h)-\PP(\beta)}{h}.
	\end{align*}
	Letting $h\downarrow 0$ and using that $\PP$ is differentiable implies that
	$$
	\lim_{N\to\infty} \sup_{M_{\beta,q}(\eps_N)}\Bigl| f_m'(\beta)-\PP'(\beta)\Bigr| = 0.
	$$
	By \eqref{zero:lem9:eq2}, denoting as before $\zeta_{\beta,m}:=\zeta_{\beta,\mu_m}$, for any $m\in (-1,1)^N$,
	$$
	f_m'(\beta) = \frac{H_N(m)}{N}+\beta\int_q^1\! (\xi'(s)-\xi'(q))\zeta_{\beta,m}(s)\,ds.
	$$
	By continuity of $\TAP^\beta(\mu,\zeta)$ in both $\mu$ and $\zeta$ and uniqueness of the minimizer, the order parameter $\zeta_{\beta,m}$ is continuous in $m$, so the same formula holds  for all $m\in [-1,1]^N$. Together with (\ref{zt:ex:eq5}) this gives
	\begin{align}
	\label{zt:ex:eq3}
	\lim_{N\to\infty} \sup_{M_{\beta,q}(\eps_N)}\Bigl|
	\frac{H_N(m)}{N}+\beta\int_q^1\bigl(\xi'(s)-\xi'(q)\bigr)\zeta_{\beta,m}(s)\,ds-
	\beta\int_0^1\!\xi'(s)\zeta_{\beta}^*(s)\,ds
	\Bigr| = 0.
	\end{align}
	
	To finish the proof of (\ref{zt:lem:eq1}), it remains to show that
	\begin{align}\label{zt:ex:eq0}
	\lim_{N\to\infty}\sup_{m\in M_{\beta,q}(\eps_N)}\int_q^1\! |\zeta_{\beta,m}(s)-\zeta_{\beta}^*(s)|\,ds=0.
	\end{align}  
	Also, \eqref{zt:lem:eq2} will follow simply by using \eqref{zt:ex:eq4} and the equality in \eqref{zt:lem:eq3} is valid directly from integration by parts. Note that, for any $m_0\in M_{\beta,q}(\eps_N)$,
	\begin{align*}
	&
	f_N(\beta)-\eps_N
	\leq \frac{\beta H_N(m_0)}{N}+\nTAP^\beta(\mu_{m_0})
	\\
	&=\frac{\beta H_N(m_0)}{N}+\nTAP^\beta(\mu_{m_0},\zeta_{\beta}^*)+\nTAP^\beta(\mu_{m_0})-\nTAP^\beta(\mu_{m_0},\zeta_{\beta}^*)\\
	&\leq \sup_{m\in S_N(q)}\Bigl(\frac{\beta H_N(m)}{N}+\nTAP^\beta(\mu_m,\zeta_{\beta}^*)\Bigr)+\nTAP^\beta(\mu_{m_0})-\nTAP^\beta(\mu_{m_0},\zeta_{\beta}^*)
	\end{align*}
	and $\nTAP^\beta(\mu_{m_0})\leq \nTAP^\beta(\mu_{m_0},\zeta_{\beta}^*).$
	These imply that
	\begin{align*}
	&\sup_{m\in M_{\beta,q}(\eps_N)}\Bigl|\nTAP^\beta(\mu_{m},\zeta_{\beta}^*)-\nTAP^\beta(\mu_{m})\Bigr|\\
	&\leq \sup_{S_N(q)}\Bigl(\frac{\beta H_N(m)}{N}+\nTAP^\beta(\mu_m,\zeta_{\beta}^*)\Bigr)-f_N(\beta)+\eps_N\to 0,
	\end{align*}
	where the a.s. convergence follows from Theorems \ref{thm:GenTAP} and \ref{Thm1label} above, the concentration of the free energy, and the Borell inequality. Now, assume on the contrary that \eqref{zt:ex:eq0} is not true. From this and the above limit, we can choose $m^N\in M_{\beta,q}(\eps_N)$ so that (by passing to a subsequence if necessary) $\mu_{m^N}\to \mu_0$ and $\zeta_{\beta,m^N}\to\zeta_0$ for some $\mu_0\in M_*$ and $\zeta_0\in \mathcal{M}_q$,
	\begin{align}\label{zt:ex:eq1}
	\begin{split}
	&\int_q^1\! |\zeta_{0}(s)-\zeta_{\beta}^*(s)|\,ds=\lim_{N\to\infty} \int_q^1\! |\zeta_{\beta,m^N}(s)-\zeta_{\beta}^*(s)|\,ds>0,
	\end{split}
	\end{align}
	and, from the continuity of $\nTAP^\beta$ on $M_*\times\mathcal{M}_q$,
	\begin{align}
	&\Bigl|\nTAP^\beta(\mu_0,\zeta_{\beta}^*)-\nTAP^\beta(\mu_0,\zeta_{0})\Bigr|
	\nonumber
	\\
	&=\lim_{N\to\infty}\Bigl|\nTAP^\beta(\mu_{m^N},\zeta_{\beta}^*)-\nTAP^\beta(\mu_{m^N},\zeta_{\beta,m^N})\Bigr|
	\label{zt:ex:eq2}
	\\
	&=\lim_{N\to\infty}\Bigl|\nTAP^\beta(\mu_{m^N},\zeta_{\beta}^*)-\nTAP^\beta(\mu_{m^N})\Bigr|
	=0.
	\nonumber
	\end{align}
	The optimality of $\zeta_{\beta,m^N},$ 
	$$
	\nTAP^\beta(\mu_{m^N})=\nTAP^\beta(\mu_{m^N},\zeta_{\beta,m^N})\leq \nTAP^\beta(\mu_{m^N},\zeta),\,\,\,\forall \zeta\in \mathcal{M}_q,
	$$ 
	yields  that
	\begin{align*}
	\nTAP^\beta(\mu_0,\zeta_0)&=\lim_{N\to\infty}\nTAP^\beta(\mu_{m^N},\zeta_{\beta,m^N})\\
	&\leq \lim_{N\to\infty}\nTAP^\beta(\mu_{m^N},\zeta)=\nTAP^\beta(\mu_0,\zeta),\,\,\,\forall \zeta\in \mathcal{M}_q.
	\end{align*}
	This means that $\zeta_0$ is a minimizer of $\nTAP^\beta(\mu_0,\,\cdot\,)$. Recall that the minimizer is unique \cite[Theorem 10]{GenTAP}, so, by (\ref{zt:ex:eq2}), $\zeta_{\beta}^*=\zeta_0$ on $[q,1].$ This contradicts \eqref{zt:ex:eq1} and finishes the proof of \eqref{zt:ex:eq0}.
\end{proof}

\section{Energy of Ancestor Measure}

In this section, we will prove Theorem \ref{add:thm1}.

\begin{proof}[Proof of Theorem \ref{add:thm1}] 
	Recall \eqref{eq:dX}  and let $X(s)=X_{0,\zeta_\beta^*}^\beta(s)$ for $s\in [0,1].$ Denote $$
	u(s)=\partial_{x}\Phi_{\zeta_{\beta,\mu}}^\beta(s,X(s))\,\,\mbox{and}\,\,v(s)=\partial_{xx}\Phi_{\zeta_{\beta,\mu}}^\beta(s,X(s)).$$ Let $\mu$ be the distribution function of the random variable $u(q).$ Note that
	\begin{align*}
	\Lambda_{\zeta_\beta^*}^\beta(q,a)&=\Phi_{\zeta_\beta^*}^\beta(q,x(a))-ax(a),
	\end{align*}
	where $x(a)$ satisfies $\partial_x\Phi_{\zeta_\beta^*}^\beta(q,x(a))=a.$ Since $\partial_x\Phi_{\zeta_\beta^*}^\beta(q,\cdot)$ is strictly increasing, it follows that if $
	a=u(q),
	$
	then $
	x(a)=X(q)$
	and hence,
	\begin{align*}
	\int\Lambda_{\zeta_\beta^*}^\beta(q,a)d\mu(a)&=\e\Phi_{\zeta_\beta^*}^\beta(q,X(q))-\e X(q)u(q)\\
	&=\e\Phi_{\zeta_\beta^*}^\beta(q,X(q))-\frac{\beta^2}{2}\int_0^q\xi''(s)\zeta_\beta^*(s)\e u(s)^2ds\\
	&\quad-\e X(q)u(q)+\frac{\beta^2}{2}\int_0^q\xi''(s)\zeta_\beta^*(s)\e u(s)^2ds\\
	&=\Phi_{\zeta_\beta^*}^\beta(0,0)-\e X(q)u(q)+\frac{\beta^2}{2}\int_0^q\xi''(s)\zeta_\beta^*(s)\e u(s)^2ds.
	\end{align*}
	Here, the middle term can be computed through
	\begin{align*}
	\e X(q)u(q)
	&=\e\Bigl(\beta^2\int_0^q\xi''(s)\zeta_\beta^*(s) u(s)ds+\beta\int_0^q\sqrt{\xi''(s)}dW_s \Bigr)u(q)\\
	&=\beta^2\int_0^q \xi''(s) \zeta_\beta^*(s)\e u^2(s)ds+\beta^2\e \Bigl(\int_0^q\sqrt{\xi''}(s)dW_s \Bigr)\Bigl(\int_0^q \sqrt{\xi''(s)}v(s)dW_s \Bigr) \\
	&=\beta^2\int_0^q \xi''(s)\bigl(\e v(s)+ \zeta_\beta^*(s)\e u^2(s)\bigr)ds.
	\end{align*}
	To handle this equation, note that $d\e u(t)^2=\beta^2\xi''(t)\e v(t)^2dt$ and $v(1)=1-u(1)^2$. These and \eqref{int} imply that
	\begin{align*}
	1-\e u(1)^2-\e v(s)&=\e v(1)-\e v(s)=-\beta^2\int_s^1\xi''(t)\zeta_\beta^*(t)\e v(t)^2dt\\
	&=-\Bigl(\e u(1)^2-\e u(s)^2\zeta_\beta^*(s)-\int_s^1\e u(t)^2d\zeta_\beta^*(t)\Bigr),
	\end{align*}
	which together with \eqref{int} leads to
	\begin{align*}
	\e v(s)+\zeta_\beta^*(s)\e u(s)^2&=1-\int_s^1\e u(t)^2d\zeta_\beta^*(t)=1-\int_s^1 td\zeta_\beta^*(t)=s\zeta_\beta^*(s)+\int_s^1\zeta_\beta^*(t)dt.
	\end{align*}
	Since
	\begin{align*}
	\beta^{-1}E_\beta(q)=\int_0^q \xi''(s)\Bigl(\int_s^1\zeta_\beta^*(t)dt\Bigr)ds,
	\end{align*}
	it follows that
	\begin{align*}
	\e X(q)u(q)&=\beta^2\int_0^q\xi''(s)\Bigl(s\zeta_\beta^*(s)+\int_s^1\zeta_\beta^*(t)dt\Bigr)ds=\beta^2\int_0^q\xi''(s)s\zeta_\beta^*(s)ds+\beta E_\beta(q).
	\end{align*}
	Consequently,
	\begin{align*}
	\nTAP^\beta(\mu)&=\Phi_{\zeta_{\beta}^*}(0,0)-\beta^2\int_0^q\xi''(s)s\zeta_\beta^*(s)ds-\beta E_\beta(q)\\
	&\quad+\frac{\beta^2}{2}\int_0^q\xi''(s)\zeta_\beta^*(s)\e u(s)^2ds-\frac{\beta^2}{2}\int_q^1\xi''(s)s\zeta_{\beta}^*(s)ds\\
	&=\mathcal{P}_\beta(\zeta_\beta^*)-\beta E_\beta(q)+\frac{\beta^2}{2}\int_0^q\xi''(s)\zeta_\beta^*(s)(\e u(s)^2-s)ds.
	\end{align*}
	Note that by the minimality of $\zeta_\beta^*,$ for any $\zeta\in \mathcal{M}_{0,1}$,
	\begin{align*}
	\frac{d}{d\theta}\mathcal{P}_\beta\bigl((1-\theta)\zeta_\beta^*+\theta\zeta\bigr)\Big|_{\theta=0^+}&=\frac{\beta^2}{2}\int_0^1\xi''(s)(\zeta(s)-\zeta_\beta^*(s))(\e u(s)^2-s)ds\geq 0.
	\end{align*}
	If, for  $s\in [0,1]$, we take 
	$$\zeta(s)=2^{-1}\zeta_\beta^*(s)1_{[0,q)}(s)+\zeta_{\beta}^*(s)1_{[q,1]}(s),$$
	then this inequality implies that
	\begin{align*}
	\int_0^q\xi''(s)\zeta_\beta^*(s)(\e u(s)^2-s)ds\leq 0.
	\end{align*}
	Hence,
	\begin{align*}
	\nTAP^\beta(\mu)&\leq \mathcal{P}_\beta(\zeta_\beta^*)-\beta E_\beta(q).
	\end{align*}
	Finally, if $q$ is in the support of $\zeta_\beta^*$, then from \cite[Equation $(46)$]{CPTAP17}, 
	\begin{align*}
	\int_0^q\xi''(s)\zeta_\beta^*(s)(\e u(s)^2-s)ds=0,
	\end{align*}
	which gives 
	\begin{align*}
	\nTAP^\beta(\mu)&=\mathcal{P}_\beta(\zeta_\beta^*)-\beta E_\beta(q).
	\end{align*}
This finishes the proof.	
\end{proof}

\section{Gradient of $\nTAP^\infty$}

In this section we establish the proof of Theorem \ref{ThmGTElab}.
Recall that by Lemma~\ref{zero:lem1}, $\beta^{-1}\nTAP^\beta(\mu)$ converges to $\nTAP^\infty(\mu)$, uniformly in $\mu\in M_*$ as $\beta\rightarrow\infty.$ Let $N\geq 1$ be fixed. Let $B$ be any compact subset of $(-1,1)^N.$ For any $m\in (-1,1)^N,$ define
\begin{align*}
f(m)=-\frac{1}{N}\Bigl(\oPsi(q_m,m_i,\gamma_{m})+m_i\xi''(q_m)\int_{q_m}^1\gamma_{m} \,ds +m_i\xi''(q_m)\Delta(m)\Bigr)_{i\leq N},
\end{align*}
where $q_m:=\sum_{i=1}^Nm_i^2/N$
and
\begin{align*}
\Delta(m)&:= \frac{1}{\xi'(1)-\xi'(q_m)}\Bigl(\nTAP^\infty(\mu_m)-\int_{q_m}^1\!(\xi'(s)-\xi'(q_m))\gamma_m(s)\,ds\Bigr).
\end{align*}
In the following, we will verify that 
\begin{align}\label{uc}
\lim_{\beta\rightarrow\infty}\sup_{m\in B}\Bigl\|\frac{1}{\beta}\nabla\nTAP^{\beta}(\mu_m)-f(m)\Bigr\|_2=0.
\end{align} 
If this is valid, this means that the gradient of $\nTAP^\infty(\mu_m)$ exists for all $m\in (-1,1)^N$ and is equal to $f(m),$ which finishes our proof. We now establish the above limit by three steps.

\medskip
{\noindent \bf Step 1.} Let $\beta_n>0$ and $m_n\in B$ be two sequences with $\beta_n\to\infty$ and $m_n\to m_0\in B$ so that 
\begin{align*}
\lim_{n\rightarrow\infty}\Bigl\|\frac{1}{\beta_n}\nabla\nTAP^{\beta_n}(m_n)-f(m_n)\Bigr\|_2=\limsup_{\beta\rightarrow\infty}\sup_{m\in B}\Bigl\|\frac{1}{\beta}\nabla\nTAP^{\beta}(m)-f(m)\Bigr\|_2.
\end{align*}
If $\zeta_{\beta_n,m_n}$ is the minimizer in the definition of $\nTAP^{\beta_n}(\mu_{m_n})$, let us denote
\begin{equation}
\zeta_n:=\zeta_{\beta_n,m_n},\,\, \gamma_n := \beta_n \zeta_n=\beta_n \zeta_{\beta_n,m_n}.
\end{equation}
By the definition of $\bN_{q_{m_n},1}$ (see \eqref{eq5.1}), if we define 
a measure $\nu_n$ on $[0,1]$ by
$$
\nu_n(A)=\int_{A}\gamma_n(s)\,ds ,
$$ 
then from \eqref{zero:lem9:eq1}, it satisfies that
\begin{align}
\int_0^1\bigl(\xi'(s)-\xi'(q_{m_n})\bigr)d\nu_n(ds)\leq \sup_{\mu\in M_*}\nTAP^\infty(\mu).
\label{unicontro}
\end{align}
From this upper bound, we can pass to a subsequence along which $\nu_n$ converges to some $\nu_0\in \bN_{q_{m_0},1}$ vaguely on $[0,1]$, where 
$$
\nu_0(A)=\int_{A}\gamma_*(s)\,ds +\Delta_*\delta_1(A)
$$
for some $\gamma_*\in \mathcal{N}_{q_{m_0},1}$ and $\Delta_*\geq 0.$
For notational clarity, we will assume throughout the rest of the proof that these hold without passing to a subsequence of $\beta_n.$ We claim that 
\begin{align}
\label{zt:add:eq5}
(\gamma_*,\Delta_*)=(\gamma_{m_0},\Delta(m_0)),
\end{align} 
where we recall \eqref{eqDefDelta} and that $\gamma_m:=\gamma_{\mu_m}$ is the minimizer as in Theorem \ref{zero:thm4}.
Indeed, by the uniform convergence of $\beta^{-1}\nTAP^\beta(\mu)$ to $\nTAP^\infty(\mu)$ and continuity of $\nTAP^\infty$,
\begin{align}
\lim_{n\to\infty}\frac{1}{\beta_n}\nTAP^{\beta_n}(\mu_{m_n})
&=\nTAP^\infty(\mu_{m_0})=\nTAP^\infty(\mu_{m_0},\gamma_{m_0}).
\label{eqOtOH}
\end{align}
On the other hand, by (\ref{zero:eq1}),
\begin{align*}
&
\lim_{n\to\infty}\frac{1}{\beta_n}\nTAP^{\beta_n}(\mu_{m_n})
=
\lim_{n\to\infty}\frac{1}{\beta_n}\nTAP^{\beta_n}(\mu_{m_n},\zeta_n)
=
\lim_{n\to\infty}\nTAP^{\infty}(\mu_{m_n},\gamma_n).
\end{align*}
For $q\in [0,1)$ and $h\in\Reals^N$, set
\begin{equation}
\nTAP^\infty(m,\gamma,h):=
\frac{1}{N}\sum_{i=1}^{N}\bigl(\Theta_{\gamma}(q_m,h_i)-m_i h_i\bigr)
-\frac{1}{2}\int_{q_m}^{1}\!s\xi''(s)\gamma(s)\,ds
\label{eqTAPfirstagGEE}
\end{equation}
so that
\begin{equation}
\nTAP^{\infty}(\mu_{m_n},\gamma_n)= \inf_{h\in\Reals^N} \nTAP^\infty(m_n,\gamma_n,h).
\label{eqTAPvariZH}
\end{equation}
If $m_0\in (-1+\eta,1-\eta)^N$, then $m_n\in (-1+\eta,1-\eta)^N$ for large $n$. It is clear from the representation (\ref{eqReprPsiInf}) and the uniform control in (\ref{unicontro}) that the minimizer $h_n$ belongs to some cube $[-L,L]^N$, where $L$ depends  only on $\eta$ and the upper bound in (\ref{unicontro}). Let us choose further subsequence along which $h_n\to h_*$. Then, using Proposition \ref{add:prop1} exactly as in the argument leading to (\ref{zt:add:eq7}), we get
$$
\lim_{n\to\infty}\nTAP^{\infty}(\mu_{m_n},\gamma_n)=
\lim_{n\to\infty}\nTAP^\infty(m_n,\gamma_n,h_n)= \nTAP^\infty(m_0,\gamma_*,h_*).
$$
By (\ref{eqOtOH}), this also equals to
\begin{equation}
\nTAP^{\infty}(\mu_{m_0},\gamma_{m_0})= \nTAP^\infty(m_0,\gamma_{m_0},h_{m_0})
\end{equation}
for some $h_{m_0}\in [-L,L]^N$. By the strict convexity of the functional (\ref{eqTAPfirstagGEE}), we must have that $\gamma_*=\gamma_{m_0}$ and $h_*=h_{m_0}$.

Note that for any $m\in [-1+\eta,1-\eta]^N$, $\TAP^\beta(\mu_m,\zeta)$ is strictly convex in $\zeta\in\mathcal{M}_{q,1}$ and that $\TAP^\beta(\mu_m,\zeta)$ is continuous in $[-1+\eta,1-\eta]^N\times \mathcal{M}_{0,1}$. From these, we see that $\zeta_{\beta,\mu_m}$ is continuous in $(\beta,m).$ As a result, from Lemma \ref{zero:lem9}, $\frac{d}{d\beta}\TAP(\mu_m)$ is continuous on $m\in [-1+\eta,1-\eta]^N$ for all $\beta>0.$ Furthermore, this derivative is nondecreasing in $\beta$ and, as $\beta\to\infty,$ it converges to $\TAP^\infty(\mu_m)$, which is a continuous function. Hence, from Dini's theorem, $\frac{d}{d\beta}\TAP^\beta(\mu_m)$ converges to $\TAP^\infty(\mu_m)$ uniformly in $m\in [-1+\eta,1-\eta]^N.$ From this, Remark \ref{rm1}, and the definition of $\nu_0$, the limit $\int_{0}^1\! \bigl(\xi'(s)-\xi'(q_{m_n})\bigr)d\nu_n(s)$ can be written in two ways, 
\begin{align*}
&\int_{q_{m_0}}^1\bigl(\xi'(s)-\xi'(q_{m_0})\bigr)\gamma_{*}(s)\,ds+\bigl(\xi'(1)-\xi'(q_{m_0})\bigr)\Delta_*=\lim_{n\to\infty}\frac{d}{d\beta}\TAP^{\beta_n}(\mu_{m_n})\\
&=\TAP^\infty(\mu_{m_0})=\int_{q_{m_0}}^1\bigl(\xi'(s)-\xi'(q_{m_0})\bigr)\gamma_{m_0}(s)\,ds+\bigl(\xi'(1)-\xi'(q_{m_0})\bigr)\Delta(m_0).
\end{align*}
Since we showed that $\gamma_*=\gamma_{m_0},$ this implies that $\Delta_{*}=\Delta(m_0)$ and finishes the proof of (\ref{zt:add:eq5}).

\medskip
{\noindent \bf Step 2.} Next, we handle the limit of the gradient of $\beta^{-1}\nTAP^\beta(\mu_m)$. Recall that from Theorem \ref{ThmGTElabbeta},
\begin{align}\label{zt:add:eq2}
\frac{1}{\beta}\nabla \nTAP^\beta(\mu_m) = -\frac{1}{N}\Bigl(\frac{1}{\beta}\oPsi_\beta(q_m,m_i,\zeta_{\beta,m})+m_i\beta\xi''(q_m)\int_{q_m}^1\zeta_{\beta,m} \,ds \Bigr)_{i\leq N}.
\end{align}
Here, the second term on the right-hand side can be handled by using the fact that $m_{n}\to m_0$, the vague convergence of $\nu_n,$ and \eqref{zt:add:eq5}, to obtain that
\begin{align}
\label{zt:add:eq4}
\int_{q_{m_n}}^1\beta_n\zeta_n(s) \,ds &=\nu_n([0,1])\to\nu_0([0,1])=\int_{q_{m_0}}^1\gamma_{m_0}(s)\,ds +\Delta(m_0).
\end{align}
Next, we treat the first term on the right-hand side of \eqref{zt:add:eq2}. Recall that for any $\zeta\in \nN,$ $a\in [-1,1]$, and $\beta>0$, we have that
\begin{align*}
\frac{1}{\beta}\Lambda_{\zeta}^{\beta}(q,a)&=\inf_{x}\Bigl(\frac{1}{\beta}\Phi_{\zeta}^{\beta}(q,\beta x)-ax\Bigr).
\end{align*}
Denote by 
\begin{align*}
x_{n,i}=\frac{1}{\beta_n}\oPsi_{\beta_n}(q_{m_n},m_{n,i},\zeta_n),\,\,\forall 1\leq i\leq N.
\end{align*}
Let us again assume without loss of generality that the following limits exist on the extended real line,
$x_i:=\lim_{n\rightarrow\infty}x_{n,i}$ for all $1\leq i\leq N.$ Then from \eqref{zt:add:eq1}, \eqref{eqLambdaRel}, and Corollary \ref{add:cor1},
\begin{align*}
&\Theta_{\nu_0}(q_{m_0},x_{i})-m_{0,i}x_{i}
=\lim_{n\to\infty}\Theta_{\nu_{n}}(q_{m_n}, x_{n,i})-m_{n,i}x_{n,i}
=\lim_{n\to\infty}\frac{1}{\beta_n}\Lambda_{\zeta_n}^{\beta_n}(q_{m_n},m_{n,i})
\\
&=\lim_{n\to\infty}\Lambda_{\nu_n}^\infty(q_{m_n},m_{n,i})
=\Lambda_{\nu_0}^\infty(q_{m_0},m_{0,i})
=\Lambda_{\gamma_{m_0}}^\infty(q_{m_0},m_{0,i})+\frac{\xi''(1)\Delta_*}{2},
\end{align*}
which means that $x_i=\oPsi(q_{m_0},m_{0,i},\gamma_{m_0}).$ Combining this with \eqref{zt:add:eq5}, \eqref{zt:add:eq2}, and \eqref{zt:add:eq4}, we arrive at
\begin{align}\label{zt:add:eq10}
\lim_{n\to\infty}\frac{1}{\beta_n}\nabla \nTAP^{\beta_n}(\mu_{m_n})&=f(m_0).
\end{align}

\medskip
{\noindent \bf Step 3.} Finally, we show that $\lim_{n\to\infty}f(m_n)=f(m_0)$ in a similar manner as the first and second steps. Once this is verified, this and \eqref{zt:add:eq10} together imply the desired uniform convergence and hence finish our proof. Recall from Remark \ref{rm1} that for each $n$, if we define the measure $\nu_n'$ on $[0,1]$ by 
\begin{align*}
\nu_n'(A)=\int_{A}\gamma_{m_n}(s)\,ds +\Delta(m_n)\delta_1(A),
\end{align*}
then
\begin{align}
\begin{split}\label{zt:add:eq11}
\int_{0}^1\bigl(\xi'(s)-\xi'(q_{m_n})\bigr)d\nu_n'(s)&=\int_{q_{m_n}}^1\bigl(\xi'(s)-\xi'(q_{m_n})\bigr)\gamma_{m_n}(s)\,ds +\bigl(\xi'(1)-\xi'(q_{m_n})\bigr)\Delta(m_n)\\
&=\nTAP^\infty(\mu_{m_n}).
\end{split}
\end{align}
Note that $\nu_n'\in \bN_{q_{m_n,1}}.$ As in Step 1, we can assume without loss of generality that $\nu_n'$ vaguely converges to some $\nu_0'\in \bN_{q_{m_0},1}$ defined as $$
\nu_0'(A):=\int_{A}\gamma_{*}'(s)\,ds +\Delta_0'\delta_1(A)
$$ 
for some $\gamma_*'\in \mathcal{N}_{q_{m_0},1}$ and $\Delta_0'\geq 0.$  We claim that 
\begin{align}\label{zt:add:eq6}
(\gamma_*',\Delta_0')=(\gamma_{m_0},\Delta(m_0)).
\end{align} 
By the argument in Step 1 above,
$$
\nTAP^\infty(\mu_{m_0},\gamma_*')=\lim_{n\to\infty}\nTAP^\infty(\mu_{m_n},\gamma_{m_n})=\lim_{n\to\infty}\nTAP^\infty(\mu_{m_n})=\nTAP^\infty(\mu_{m_0}).
$$
Hence, the uniqueness of the minimizer forces $\gamma_*'=\gamma_{m_0}.$ On the other hand, the vague convergence of $\nu_n'$ to $\nu_0'$ and \eqref{zt:add:eq11} imply that
\begin{align*}
&\int_{q_{m_0}}\bigl(\xi'(s)-\xi'(q_{m_0})\bigr)\gamma_{m_0}(s)\,ds +\bigl(\xi'(1)-\xi'(q_{m_0})\bigr)\Delta_0'\\
&=\int_{0}^1\bigl(\xi'(s)-\xi'(q_{m_0})\bigr)\nu_0'(s)=\lim_{n\to\infty}\int_{0}^1\bigl(\xi'(s)-\xi'(q_{m_n})\bigr)d\nu_n'(s)=\nTAP^\infty(\mu_{m_0}),
\end{align*}
which means that $\Delta_0'=\Delta(m_0).$ These complete the proof of \eqref{zt:add:eq6}. Now, from \eqref{zt:add:eq6},
\begin{align*}
&\lim_{n\to\infty}m_{n,i}\xi''(q_{m_n})\int_{q_{m_n}}^1\gamma_{m_n} \,ds +m_{n,i}\xi''(q_{m_n})\Delta(m_n)
=\lim_{n\to\infty}m_{n,i}\xi''(q_{m_n})\nu_n'([0,1])
\\
&
=m_{0,i}\xi''(q_{m_0})\nu_0'([0,1])
=m_{0,i}\xi''(q_{m_0})\int_{q_{m_0}}^1\gamma_{m_0} \,ds  +m_{0,i}\xi''(q_{m_0})\Delta(m_0).
\end{align*}
Furthermore, following a similar argument as we handled the first term on the right-hand side of \eqref{zt:add:eq2} in the second step, it can also be obtained that 
\begin{align*}
\lim_{n\to\infty}\oPsi(q_{m_n},m_{n,i},\gamma_{m_n})=\oPsi(q_{m_0},m_{0,i},\gamma_{m_0}).
\end{align*}
This together with the above limit gives that $\lim_{n\to\infty}f(m_n)=f(m_0)$ and this completes our proof.
\qed

\bibliographystyle{plain}
\bibliography{bibliography}

\end{document}